\documentclass[a4paper,10pt]{article}

\usepackage{geometry}
\usepackage{amsmath}
\usepackage{amsfonts,color,amssymb,mathrsfs,stmaryrd}
\usepackage{amsthm}
\usepackage{amssymb}
\usepackage{url}
\usepackage{subcaption}
\usepackage{bera}
\usepackage{array}

\usepackage[normalem]{ulem}

\usepackage{graphicx}
\usepackage{wrapfig}
\usepackage{mathrsfs}
\usepackage{bbm}
\usepackage{enumitem}
\usepackage{tikz}
\usetikzlibrary{angles, quotes, calc, decorations.pathreplacing}

\newcommand{\p}{\mathbb{P}}
\newcommand{\E}{\mathbb{E}}

\newcommand{\tah}{\tanh\left(\frac{h}{2}\right)}
\newcommand{\tahq}{\tanh^2\left(\frac{h}{2}\right)}

\renewcommand{\H}{{\cal P}}

\numberwithin{equation}{section}

\newtheorem{theorem}{Theorem}\numberwithin{theorem}{section}
\newtheorem{coro}[theorem]{Corollary}

\newtheorem{prop}[theorem]{Proposition}
\newtheorem{lemma}[theorem]{Lemma}

\newtheorem{rema}{Remark}

\def\HH{\bar{H}}
\def\E{\mathbb{E}}

\def\H{\mathbb{H}}
\def\0{{\bf 0}}

\def\R{\mathbb{R}}
\def\PP{\mathbb{P}}

\renewcommand{\E}{\mathbb E \,}

\newcommand{\atanh}{\mathrm{artanh}}

\def\beqn{\begin{equation}}
	\def\eeqn{\end{equation}}
\def\be{\begin{equation}}
	\def\ee{\end{equation}}

\def\R{\mathbb{R}}

\def\dint{\textup{d}}

\def\qed{\hfill\hbox{${\vcenter{\vbox{
					\hrule height 0.4pt\hbox{\vrule width 0.4pt height 6pt
						\kern5pt\vrule width 0.4pt}\hrule height 0.4pt}}}$}}

\def\dint{\mathrm{d}}
\def\HH{\mathbb{H}}

\renewcommand{\thefootnote}{\fnsymbol{footnote}}

\usepackage{titlesec}

\titleformat*{\section}{\normalfont\large\bfseries}
\titleformat*{\subsection}{\normalfont\bfseries}

\date{\vspace{-0.95cm}}

\makeatletter
\let\@fnsymbol\@alph
\makeatother

\begin{document}
	\title{\bfseries Seeing Through Hyperbolic Space: \\ Visibility for $\lambda$-Geodesic Hyperplanes}
	
	\author{%
		Zakhar Kabluchko\footnotemark[1]
		\and Vanessa Mattutat\footnotemark[2]
		\and  Christoph Thäle\footnotemark[3]
	}
	
	\date{}
	\renewcommand{\thefootnote}{\fnsymbol{footnote}}

	\footnotetext[1]{%
		University of Münster, Germany. Email: zakhar.kabluchko@uni-muenster.de
	}
	\footnotetext[2]{%
		Ruhr University  Bochum, Germany. Email: vanessa.mattutat@rub.de
	}
	\footnotetext[3]{%
		Ruhr University  Bochum, Germany. Email: christoph.thaele@rub.de
	}

	\maketitle
	\begin{abstract}\noindent
		We study visibility from a fixed point in the presence of a Poisson process of $\lambda$--geodesic hyperplanes in a $d$-dimensional hyperbolic space.  The family of $\lambda$--geodesic hyperplanes interpolates between totally geodesic hyperplanes and horospheres. Our main result establishes a universality principle for this model: we prove that the fundamental visibility properties are invariant with respect to the parameter $\lambda\in[0,1]$. Namely, there is a critical intensity $\gamma_{\mathrm{crit}}>0$ such that the visible region is unbounded with positive probability for $\gamma < \gamma_{\mathrm{crit}}$ and almost surely bounded for $\gamma > \gamma_{\mathrm{crit}}$. For $d=2$ we establish almost sure boundedness also at criticality. The value for $\gamma_{\mathrm{crit}}$ is explicit and does not depend on $\lambda$. In the bounded phase, we show that the mean visible volume is identical with the known formula for  $\lambda=0$.  The key integral-geometric step is an explicit computation showing that the  measure of $\lambda$-geodesic hyperplanes hitting a geodesic segment is a linear function of the length of the segment, independent  of~$\lambda$.
		
		
	\end{abstract}

	
	\section{Introduction}
		
	Random tessellations generated by Poisson hyperplane processes in Euclidean space form a classical and well-studied topic in stochastic geometry.  In $\mathbb{R}^d$, a stationary Poisson hyperplane process gives rise to a mosaic of convex cells, and many aspects of its geometry have been analyzed in detail, including the cell containing the origin, see \cite{HugSchneiderBook} and Chapter 10 of \cite{SW08}.
	
	When passing from Euclidean space to hyperbolic space $\mathbb{H}^d$ of constant negative curvature  $-1$, the situation becomes richer, since hyperbolic geometry admits several natural analogues of Euclidean hyperplanes. Besides totally geodesic hyperplanes, which are isometric copies of $\mathbb{H}^{d-1}$, one encounters equidistant hypersurfaces, which are defined as the set of points at fixed distance from a given geodesic hyperplane. A further limiting case is provided by horospheres, which arise as limits of equidistant hypersurfaces whose base hyperplane recedes towards the ideal boundary of $\mathbb{H}^d$. These three families can be unified by the notion of
	\emph{$\lambda$--geodesic hyperplanes} as introduced in \cite{Solanes}. For $\lambda \in [0,1]$, a
	$\lambda$--geodesic hyperplane is a complete totally umbilical hypersurface in $\mathbb{H}^d$ with normal curvature $\lambda $. Thus, $\lambda = 0$ corresponds to totally geodesic hyperplanes, which have vanishing intrinsic curvature, $0 < \lambda < 1$ yields equidistant hypersurfaces, whose curvature depends continuously on~$\lambda$, and $\lambda = 1$ gives horospheres.
	
	Every $\lambda$--geodesic hyperplane separates $\mathbb{H}^d$ into two connected components.  As we are not interested in all such hypersurfaces, we introduce the space $\mathrm{Hyp}_\lambda^o$ of all $\lambda$--geodesic hyperplanes for which the distinguished point
	$o \in \mathbb{H}^d$, which we refer to as the origin, does not lie on the convex side.  Let $\nu_\lambda$ denote the natural isometry-invariant measure on the space of all $\lambda$-geodesic hyperplanes, as introduced in Section \ref{sec:lambda-geodesic_hyperplanes} below,  and $\nu_\lambda^o$ its restriction to $\mathrm{Hyp}_\lambda^o$. For an intensity parameter $\gamma > 0$, we define $\eta_{\gamma,\lambda}$
	to be a Poisson point process on $\mathrm{Hyp}_\lambda^o$ with intensity measure $\gamma \nu_\lambda^o$.  Each hyperplane in $\eta_{\gamma,\lambda}$ acts as a random obstacle to visibility from the reference point~$o$. For a point $x \in \mathbb{H}^d$, we say that $x$ is visible from~$o$ if the geodesic segment $[o,x]$ connecting $o$ with $x$ does not intersect any $H \in \eta_{\gamma,\lambda}$.
	This leads to the random open set
	\begin{equation}\label{eq:DefVisibilitySet}
	Z_{\gamma,\lambda,d}
	:= \{\, x \in \mathbb{H}^d :
	[o,x] \cap H = \varnothing
	\text{ for all } H \in \eta_{\gamma,\lambda}
	\,\},
	\end{equation}
	which we call the \emph{visibility region} of~$o$. Equivalently, for $\lambda=0$, $Z_{\gamma,0,d}$ is the connected component of~$o$ in the complement of the union of all hyperplanes in $\eta_{\gamma,0}$. Simulations of Poisson processes of $\lambda$-geodesic hyperplanes for different values of $\lambda$ together with the corresponding visibility regions are shown in Figure \ref{fig: simulation}.

		\begin{figure}
		\centering
		\begin{minipage}{.3\linewidth}
			\hspace*{-0.4cm}\includegraphics[scale=0.09]{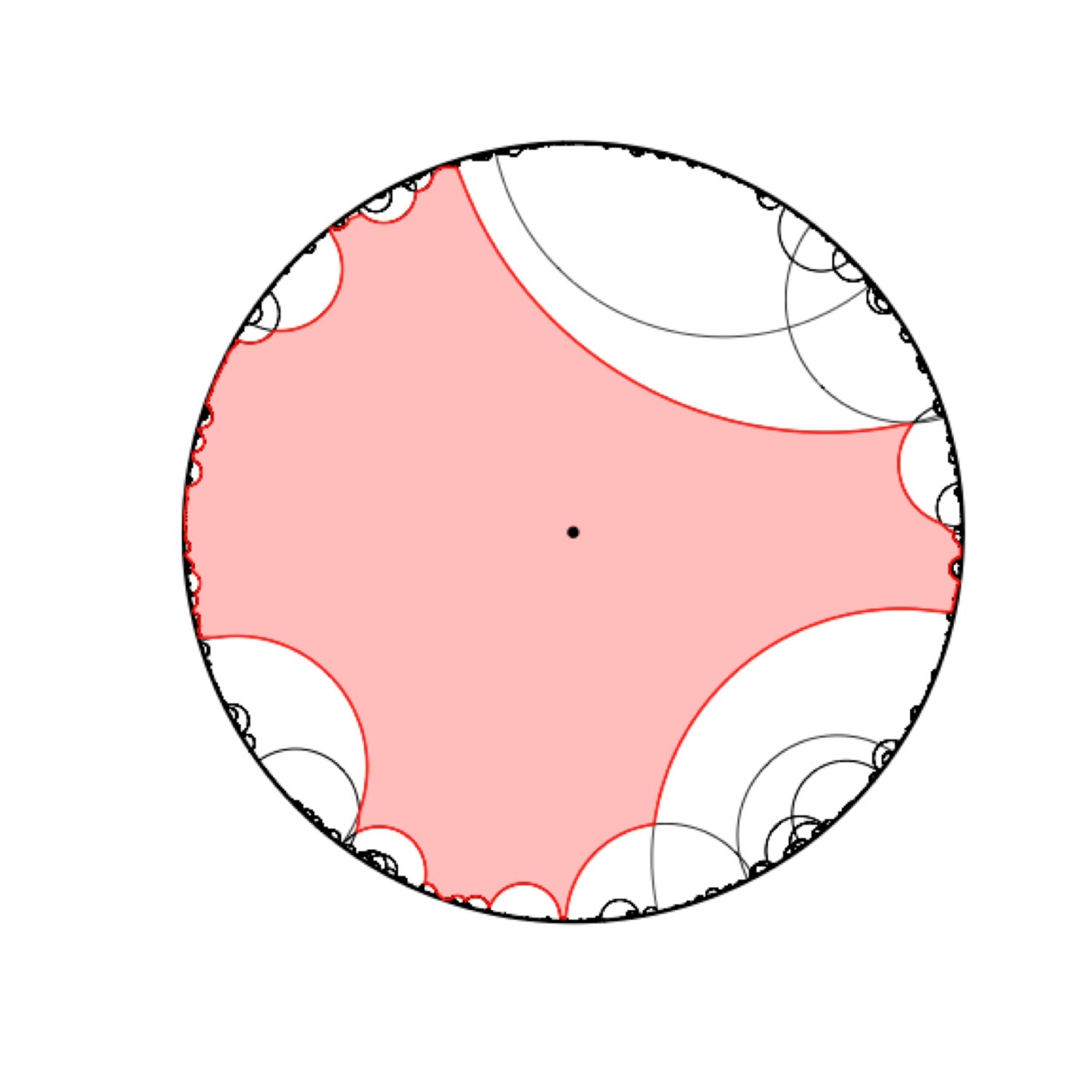}
			\subcaption{$\lambda=0, \gamma=2<\gamma_{\mathrm{crit}}$}
		\end{minipage}\quad
		\begin{minipage}{.3\linewidth}
			\hspace*{-0.4cm}	\includegraphics[scale=0.09]{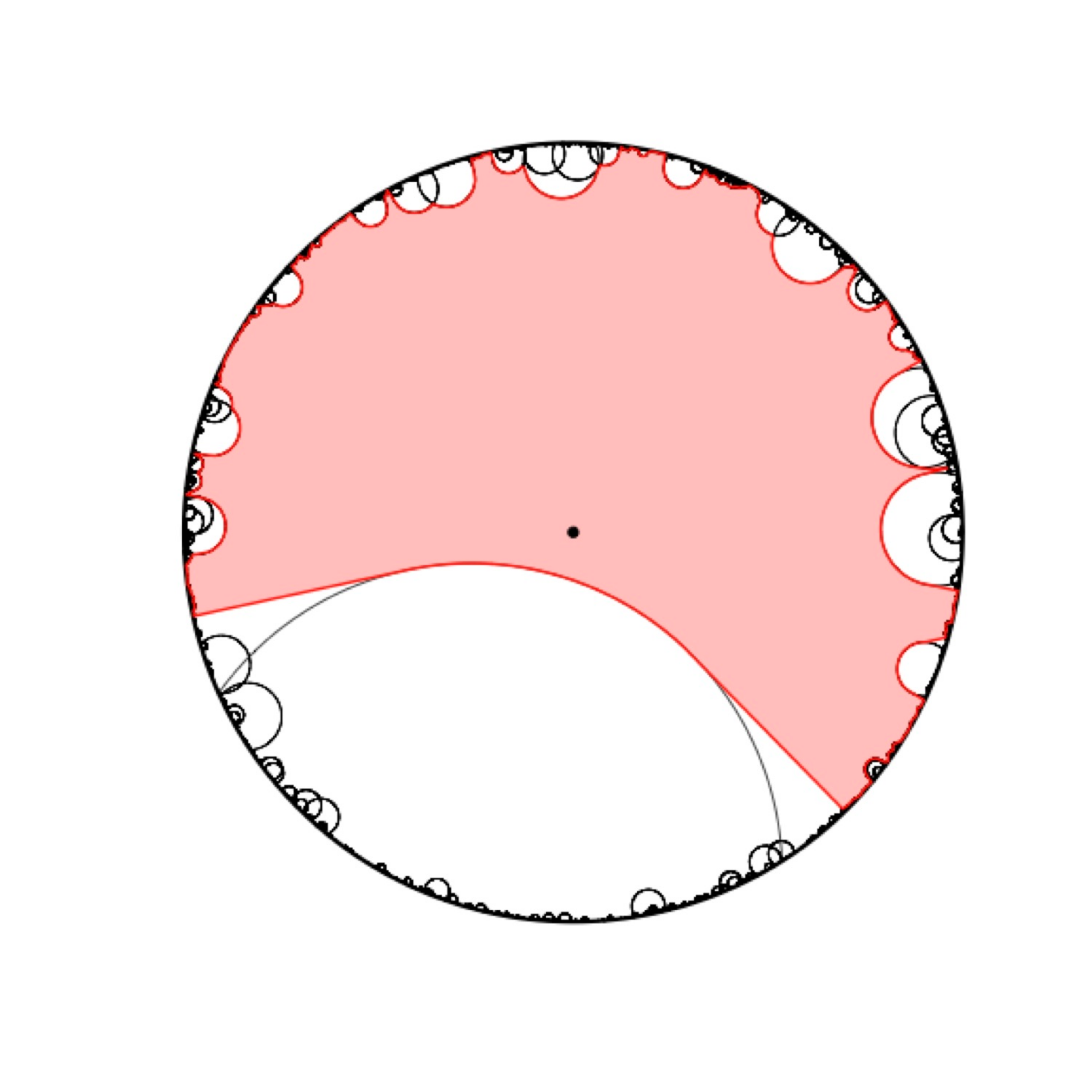}
			\subcaption{$\lambda=0.5, \gamma=2<\gamma_{\mathrm{crit}}$}
		\end{minipage}\quad
		\begin{minipage}{.3\linewidth}
			\hspace*{-0.4cm}\includegraphics[scale=0.09]{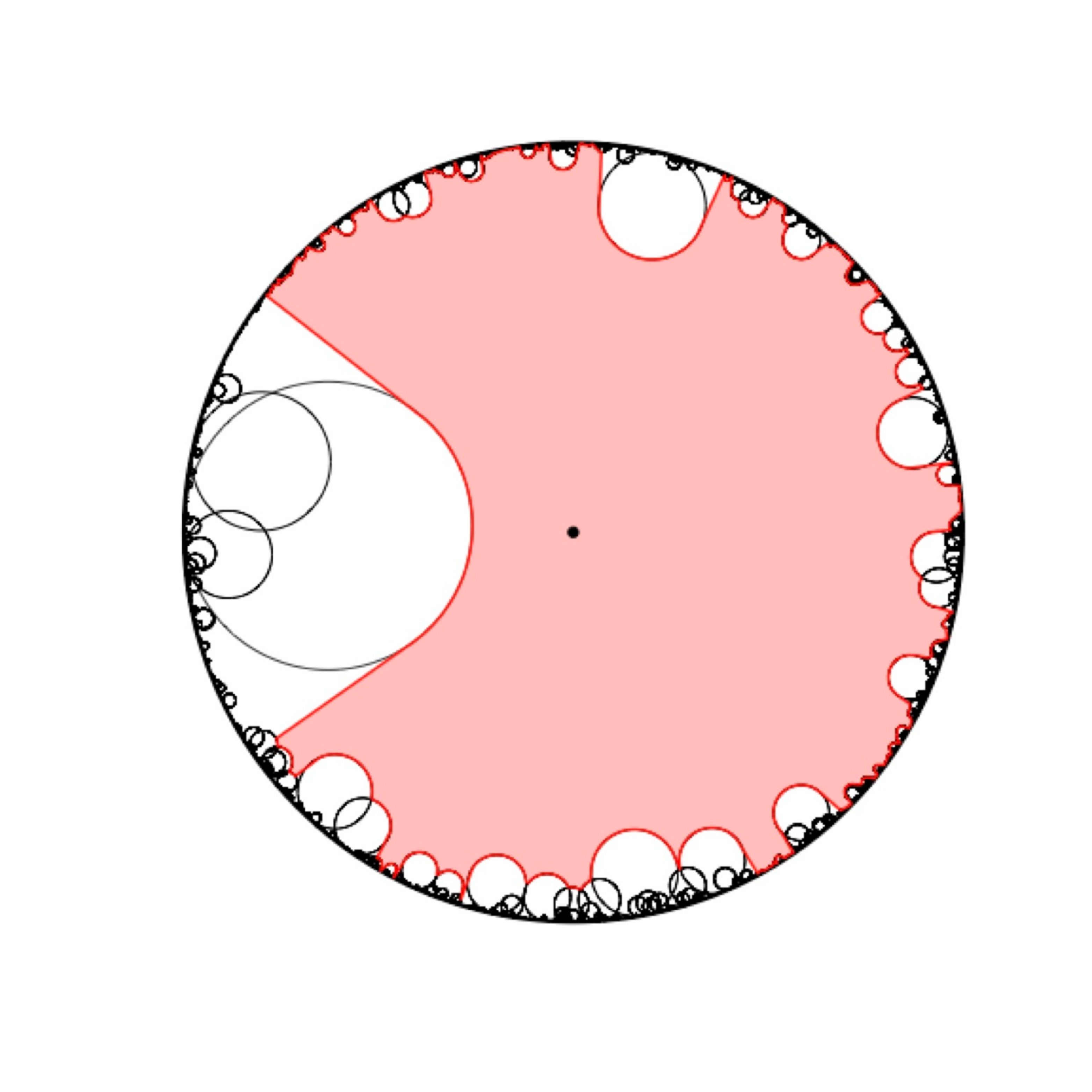}
			\subcaption{$\lambda=1, \gamma=2<\gamma_{\mathrm{crit}}$}
		\end{minipage}
		\begin{minipage}{.3\linewidth}
			\hspace*{-0.4cm}\includegraphics[scale=0.09]{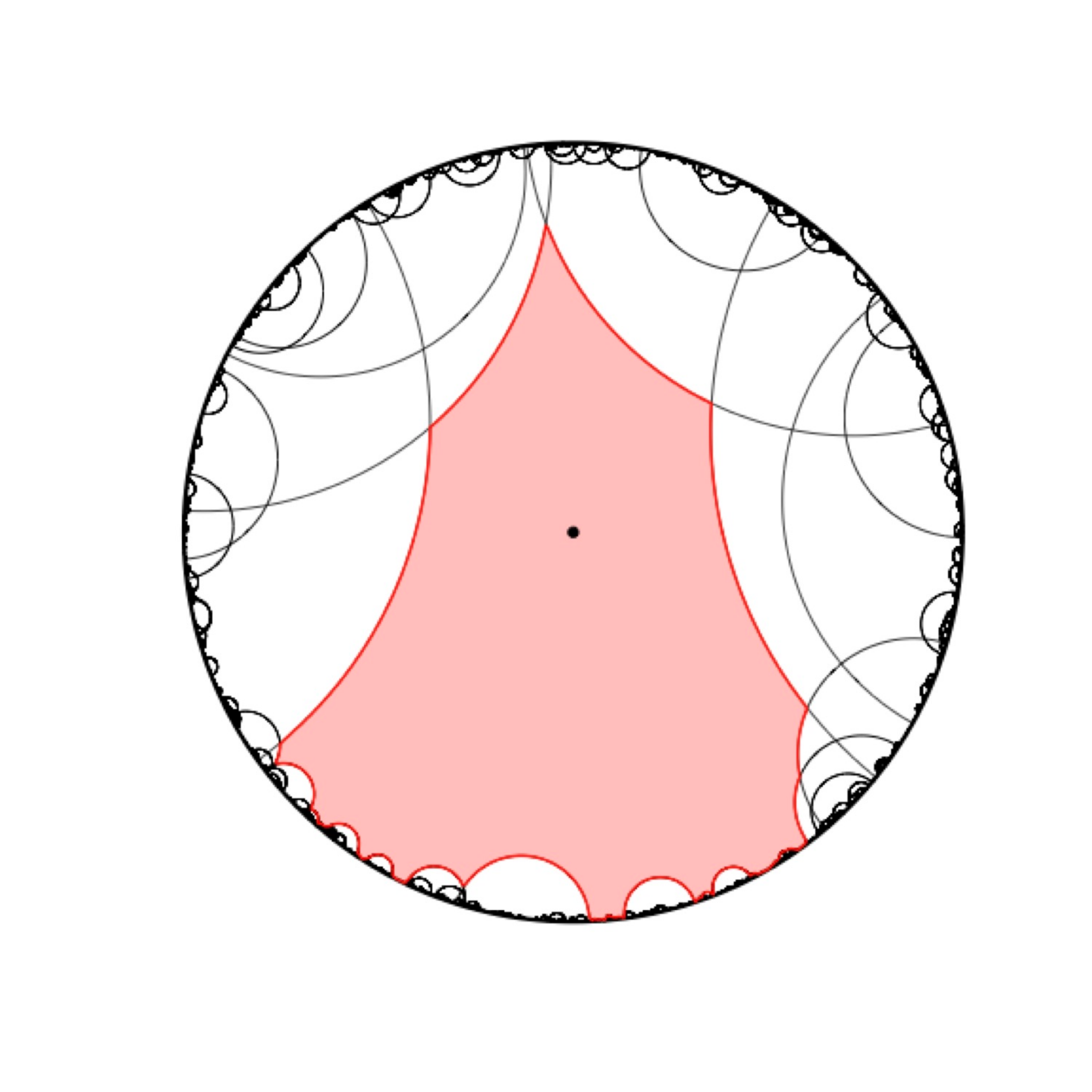}
			\subcaption{$\lambda=0, \gamma=\pi=\gamma_{\mathrm{crit}}$}
		\end{minipage}\quad
		\begin{minipage}{.3\linewidth}
			\hspace*{-0.4cm}	\includegraphics[scale=0.09]{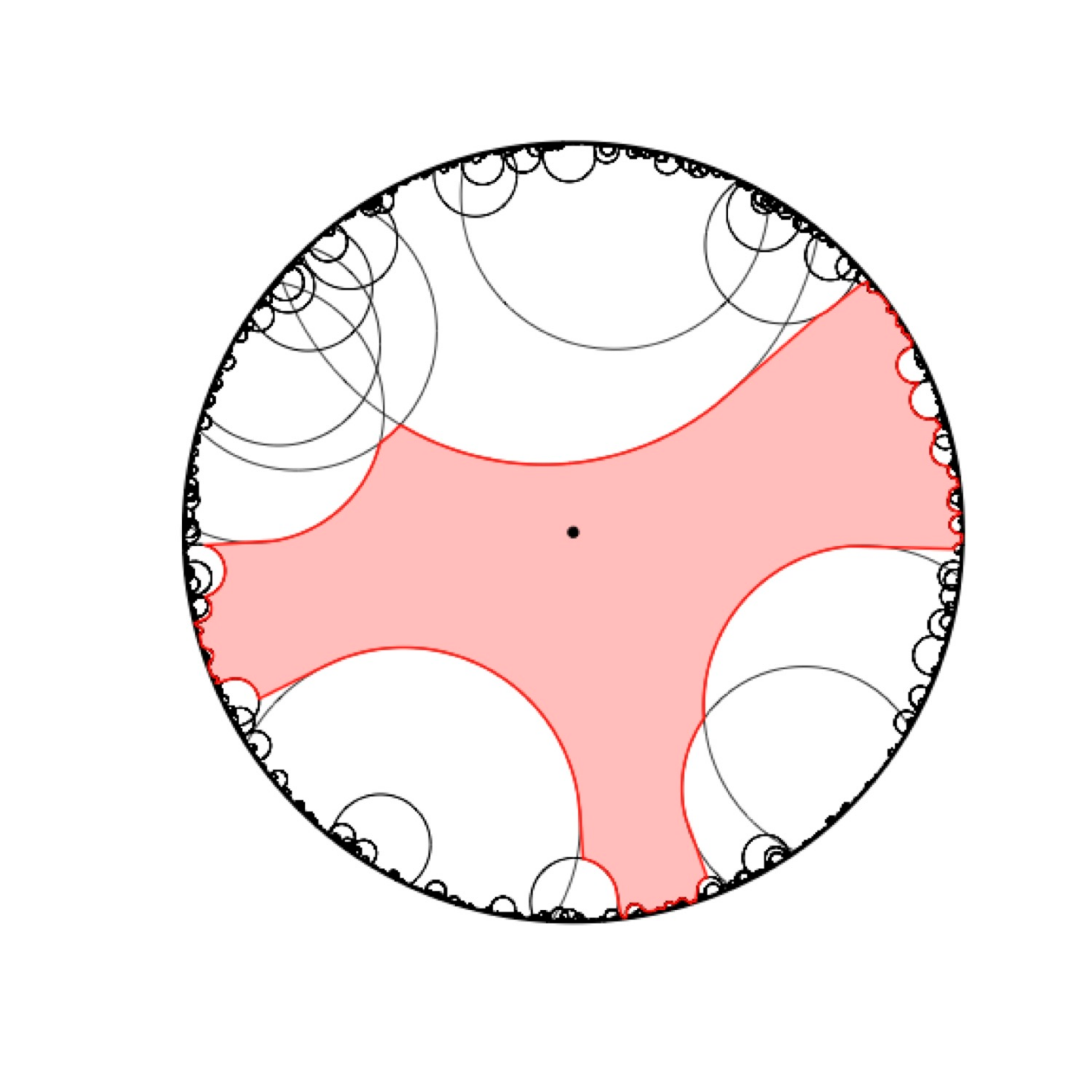}
			\subcaption{$\lambda=0.5, \gamma=\pi=\gamma_{\mathrm{crit}}$}
		\end{minipage}\quad
		\begin{minipage}{.3\linewidth}
			\hspace*{-0.4cm}\includegraphics[scale=0.09]{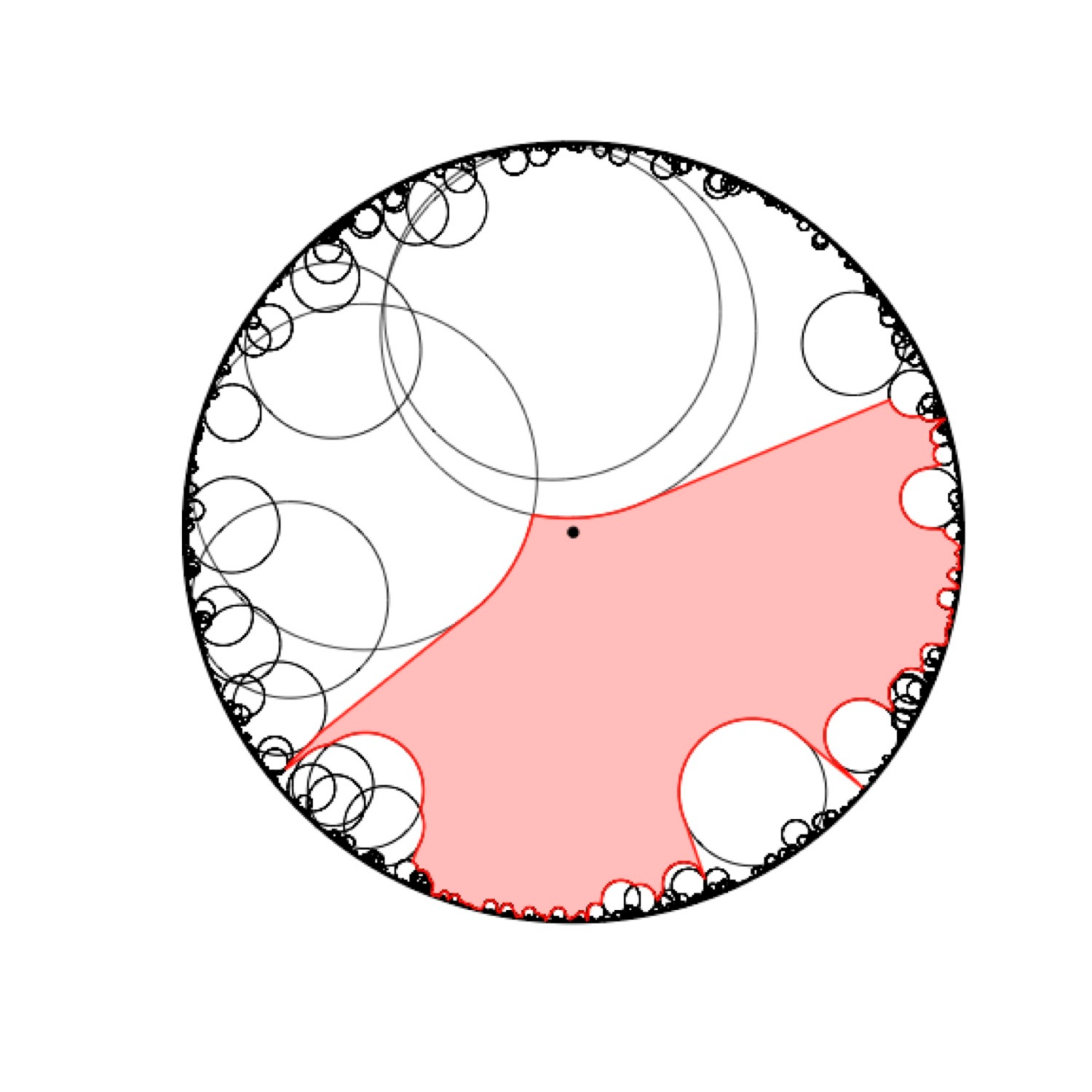}
			\subcaption{$\lambda=1,\gamma=\pi=\gamma_{\mathrm{crit}}$}
		\end{minipage}
		\begin{minipage}{.3\linewidth}
			\hspace*{-0.4cm}\includegraphics[scale=0.09]{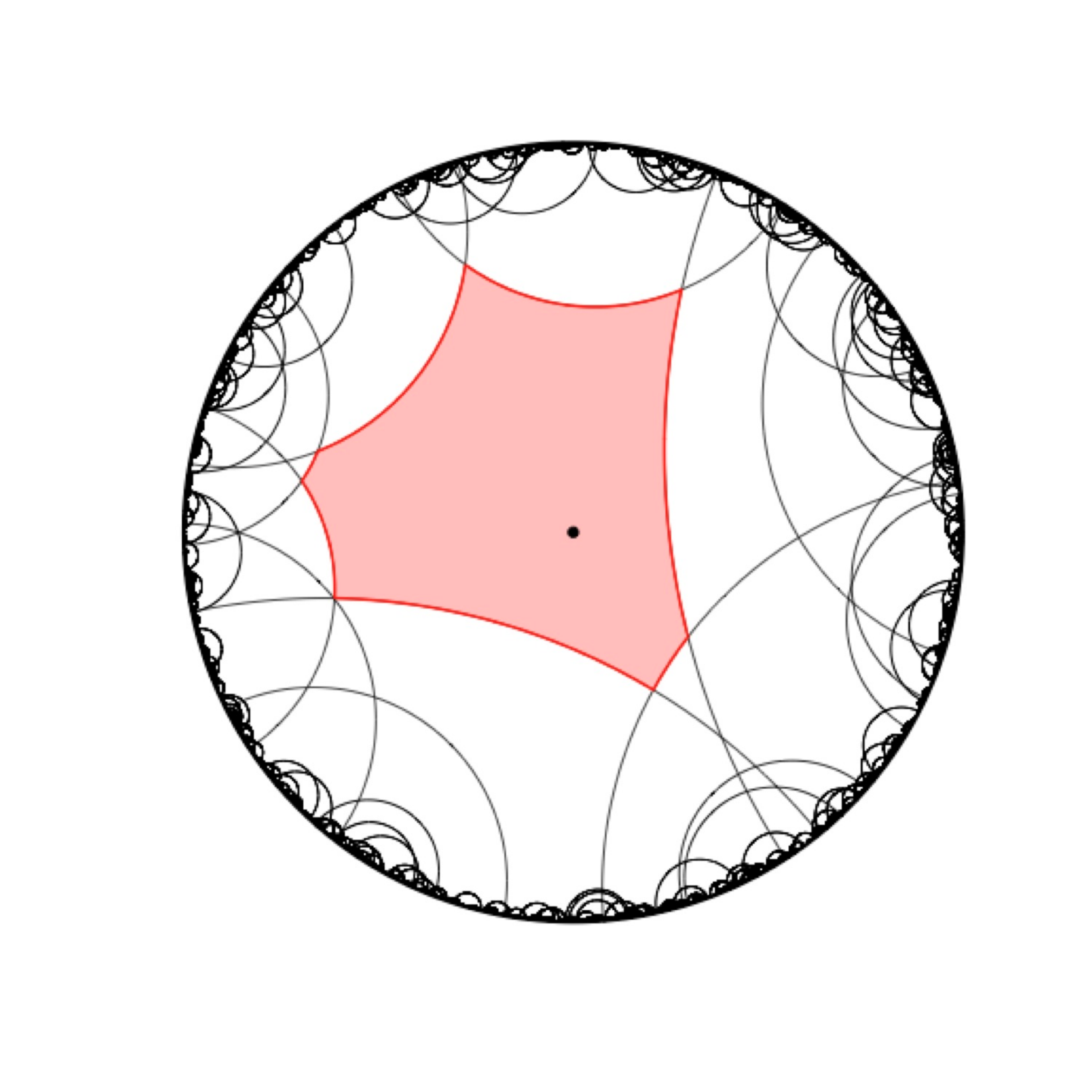}
			\subcaption{$\lambda=0, \gamma=7>\gamma_{\mathrm{crit}}$}
		\end{minipage}\quad
		\begin{minipage}{.3\linewidth}
			\hspace*{-0.4cm}	\includegraphics[scale=0.09]{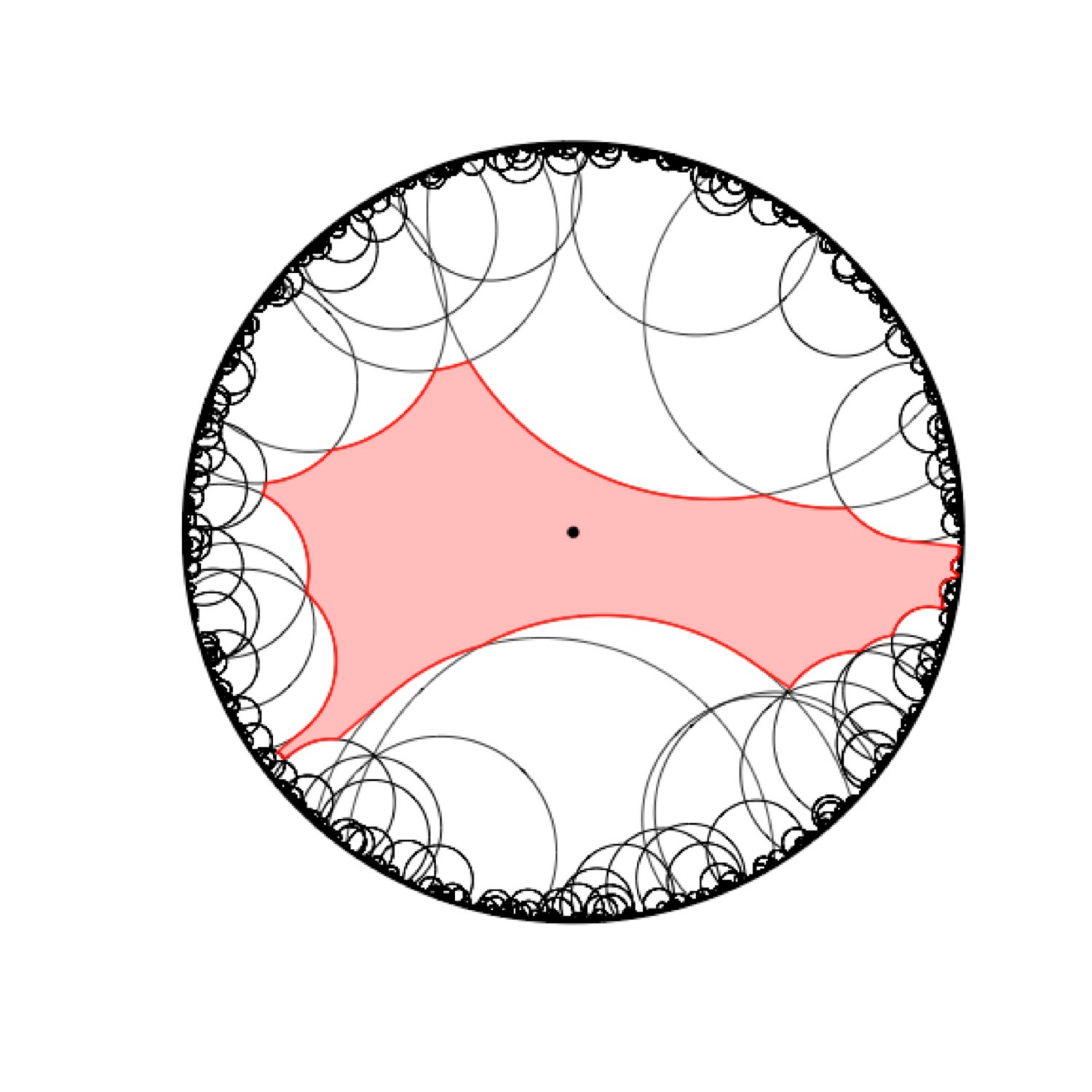}
			\subcaption{$\lambda=0.5, \gamma=7>\gamma_{\mathrm{crit}}$}
		\end{minipage}\quad
		\begin{minipage}{.3\linewidth}
			\hspace*{-0.4cm}\includegraphics[scale=0.09]{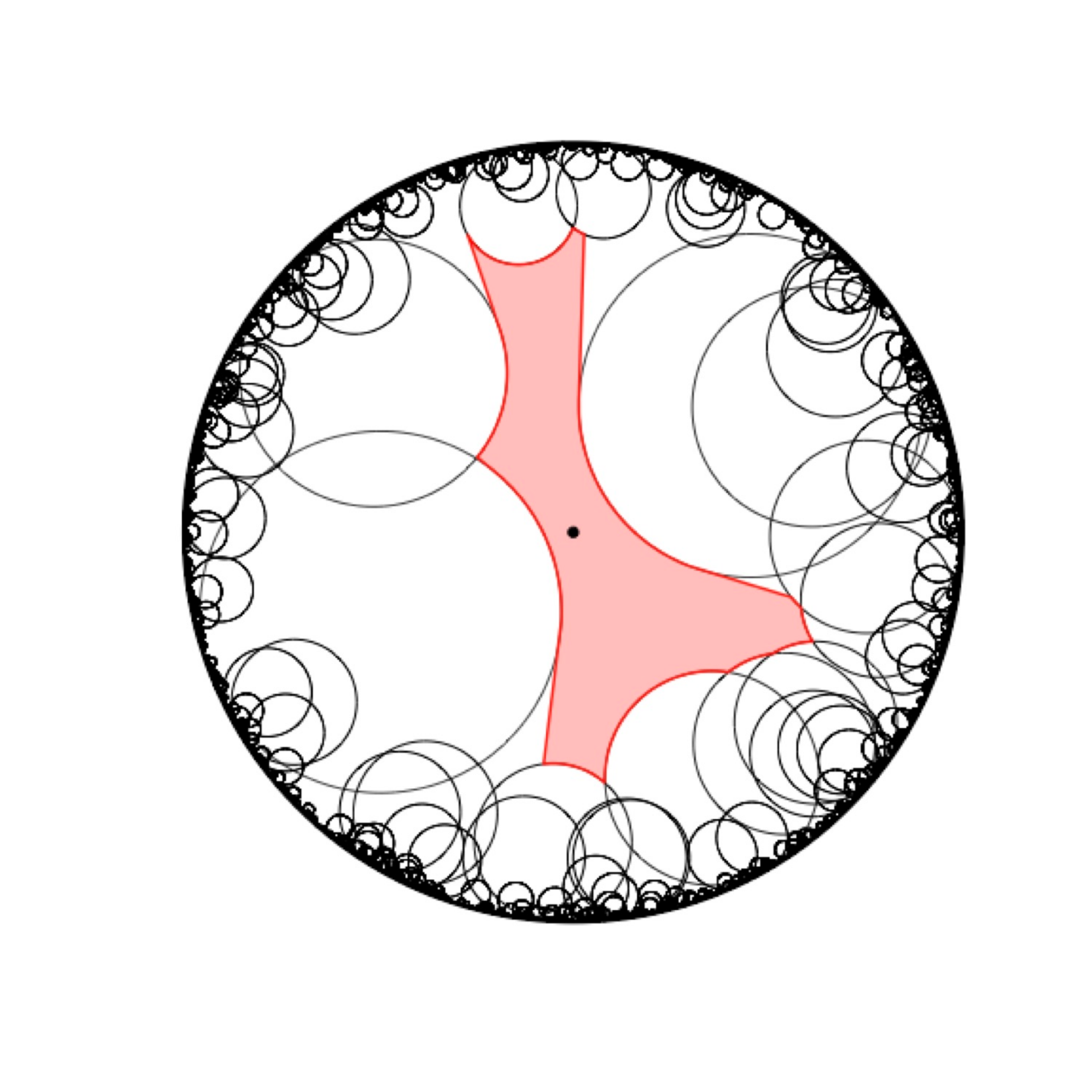}
			\subcaption{$\lambda=1, \gamma=7>\gamma_{\mathrm{crit}}$}
		\end{minipage}
		\caption{Simulation of a Poisson process of $\lambda$-geodesic hyperplanes in $\mathrm{Hyp}_\lambda^o$ with the corresponding visibility set in red in the Poincar\'e ball model. Note that for $d=2$, $\gamma_{\mathrm{crit}}=\pi$.}
		\label{fig: simulation}
	\end{figure}
	
	The case $\lambda = 0$, corresponding to totally geodesic hyperplanes, has been studied in detail in previous works \cite{BJST,BuehlerGusakovaRecke,BuehlerHugThaele,GKT22,PorretBlanc}.  These papers show that the visibility region $Z_{\gamma,0,d}$ may be unbounded, and that its behaviour undergoes a sharp phase transition as the intensity~$\gamma$ varies.  More precisely, there exists the critical value
	\begin{equation}\label{eq:GammaCrit}
	\gamma_{\mathrm{crit}}:= \sqrt{\pi}(d-1)^2{\Gamma({d-1\over 2})\over\Gamma({d\over 2})}
	\end{equation}
	such that
	\begin{itemize}
		\item if $\gamma < \gamma_{\mathrm{crit}}$, then
		$Z_{\gamma,0,d}$ is unbounded with strictly positive probability,
		\item if $\gamma \geq \gamma_{\mathrm{crit}}$, then
		$Z_{\gamma,0,d}$ is almost surely bounded.
	\end{itemize}
	The critical case $\gamma = \gamma_{\mathrm{crit}}$ was treated in \cite{GKT22,PorretBlanc} for $d=2$ and in \cite{BuehlerGusakovaRecke} for general dimensions. Thus, for small intensities, there remains a positive chance of seeing arbitrarily far from the origin, whereas for large intensities the random hypersurfaces almost surely form a finite ``cocoon'' around~$o$.  In the latter regime, the expected hyperbolic volume of the bounded visibility region has been computed in \cite[Theorem 7.1]{BuehlerHugThaele}, yielding an explicit expression in terms of the dimension~$d$ and the intensity parameter~$\gamma$.

	In the present work we investigate the same questions for general
	$\lambda$-geodesic hyperplanes.  One might expect the behaviour of the visibility region to depend sensitively on the parameter $\lambda$, since varying $\lambda$ interpolates between different geometric objects. Surprisingly, our results show that the situation is completely rigid with respect to~$\lambda$.  The critical intensity at which the transition from unbounded to bounded
	visibility occurs is \emph{identical} for all $\lambda \in [0,1]$, and in
	the bounded regime the expected volume of the visibility region
	$Z_{\gamma,\lambda,d}$ agrees exactly with the expression known from the case $\lambda = 0$.  In other words, neither the existence of infinite sightliness nor the size of the bounded visibility region depends in any way on the geometric nature of the obstructing hypersurfaces.  Our main findings are summarized in the following result.
	
	\begin{theorem}\label{thm:MainIntro}
	Fix $d\geq 2$, $\gamma>0$ and $0\leq\lambda\leq 1$. Consider the visibility region $Z_{\gamma,\lambda,d}$ of $o$ in a Poisson process of $\lambda$-geodesic hyperplanes in $\mathbb{H}^d$ with intensity measure $\gamma\nu_\lambda^o$.
	\begin{itemize}
	\item[(a)] Let $\gamma_{\mathrm{crit}}$ be as in \eqref{eq:GammaCrit}. If $\gamma<\gamma_{\mathrm{crit}}$, then $Z_{\gamma,\lambda,d}$ is unbounded with strictly positive probability, whereas if $\gamma>\gamma_{\mathrm{crit}}$, then $Z_{\gamma,\lambda,d}$ is almost surely bounded. Moreover, if $d=2$, then $Z_{\gamma_{\mathrm{crit}},\lambda,2}$ is almost surely bounded.
	
	\item[(b)] If $\gamma>\gamma_{\mathrm{crit}}$ then
	\begin{equation}\label{eq:MeanVolume}
		\mathbb{E}{\rm vol}_d(Z_{\gamma,\lambda,d})=\mathbb{E}{\rm vol}_d(Z_{\gamma,0,d}) = \pi^{\frac{d-1}{2}} \Gamma\left(\frac{d+1}{2}\right){\Gamma\big({\gamma_d^\ast-d+1\over 2}\big)\over\Gamma\big({\gamma_d^\ast+d+1\over 2}\big)}
	\end{equation}
	with $\gamma_d^\ast := \gamma\,{\Gamma({d\over 2})\over 2\sqrt{\pi}\Gamma({d+1\over 2})}$. On the other hand, $\mathbb{E}{\rm vol}_d(Z_{\gamma,\lambda,d})=+\infty$ for all $\gamma\leq\gamma_{\mathrm{crit}}$.
	\end{itemize}
	\end{theorem}

	The proofs of the results in (a) and (b) rely on two complementary ingredients.
	For the existence of a phase transition, we adapt a covering criterion due to Hoffmann--J\o rgensen \cite{H73}. This method parallels the approach in \cite{GKT22} used in the case $\lambda = 0$, and it  shows that the critical threshold does not depend on the geometric nature of the
	hyperplanes, i.e., on $\lambda$. The concrete computations rely on a subtle cancellation, which eliminates the $\lambda$-dependence from the relevant density terms.
	The computation of the expected volume in the bounded phase reduces to an integral-geometric question: one must determine the measure $\nu_\lambda^o$ of $\lambda$--geodesic hyperplanes that intersect a fixed geodesic segment of length~$h>0$.
	For $\lambda = 0$, this quantity can be determined using the classical Crofton formula in hyperbolic space. In fact, if $x_h$ is a point with hyperbolic distance $h$ from $o$ we have
	\begin{equation}\label{eq:IndicatorToEulerChar}
	\int_{{\rm Hyp}_0^o}{\bf 1}\{H\cap[o,x_h]\neq\varnothing\}\,\nu_0^o(\dint H) = \int_{{\rm Hyp}_0^o}\chi(H\cap[o,x_h])\,\nu_0^o(\dint H) = {\Gamma({d\over 2})\over 2\sqrt{\pi}\Gamma({d+1\over 2})}\,h,
	\end{equation}
	where $\chi$ denotes the Euler characteristic, see \cite[Section 7]{BuehlerHugThaele}.  However, the Crofton argument breaks down for $\lambda>0$, since a
	$\lambda$--geodesic hyperplane may intersect a geodesic segment either once or twice. This implies that for all such $\lambda$ the indicator function cannot be replaced by the Euler characteristic and it seems that no direct integral-geometric relation is available.  We overcome this  by introducing an explicit parametrization of $\lambda$--geodesic hyperplanes and carrying out the computation directly.  The resulting expression turns out to be independent of~$\lambda$ and again proportional to the length~$h$. This invariance is a surprising integral-geometric fact of independent interest and is ultimately responsible for the $\lambda$--independence of the expected volume as described in Theorem \ref{thm:MainIntro}.
	
	Our findings are placed within the broader context of stochastic geometry in hyperbolic space, a field that has seen significant recent activity. This area encompasses a rich variety of random structures adapted to negative curvature. A selection of recent results includes the asymptotic geometry of random polytopes \cite{BesauThaele,FodorGruenfelder}, the investigation of the Boolean model \cite{BJST,BuehlerHugThaele,HugLastSchulte,Tykesson} and the study of hyperbolic random graphs, which connects geometric probability with network science \cite{BlaesiusETal,FountulakisEtal,KrioukovEtal}. The investigation of Poisson point processes and their associated Voronoi tessellations remained a central theme, with foundational works including \cite{Isokawa}. A particular focus has recently developed around ideal Poisson--Voronoi tessellations and their applications, where the generating points reside on the boundary at infinity \cite{Achille,AchilleEtal,BudzinskiEtal,Fraczyk}. Our work complements this diverse landscape by focusing on the geometry induced by Poisson processes of $\lambda$-geodesic hyperplanes, uncovering a surprising universality in visibility properties.
	
	\medspace
	
	The paper is organized as follows. Section \ref{sec:Background} collects the necessary background on hyperbolic geometry to make the paper self-contained. In Section \ref{sec:Visibility}, we establish the visibility phase transition, proving Theorem \ref{thm:MainIntro}(a). Section \ref{sec:IntegralGeometrySegment} is dedicated to the core integral-geometric mechanism: we prove that the Crofton-type identity \eqref{eq:IndicatorToEulerChar} holds universally for all $\lambda\in[0,1]$. Finally, we apply this result in Section \ref{sec:ExpectedVolume} to derive the mean visible volume formula in Theorem \ref{thm:MainIntro}(b).
	
	\section{Background material}\label{sec:Background}
	
	\subsection{Hyperbolic space and the Poincar\'e ball model}

	For $d \in \mathbb{N}$ with $d \geq 2$, let $\mathbb{H}^d$ denote the
	$d$--dimensional hyperbolic space of constant curvature $-1$.
	One convenient realisation of $\mathbb{H}^d$ is given by the
	Poincar\'e ball model, in which the underlying space is the open unit ball
	$\mathbb{B}^d := \{\, y \in \mathbb{R}^d : \|y\| < 1 \,\}$, equipped with the
	Riemannian metric $\frac{4}{(1 - \|y\|^2)^2}\, \mathrm{d}y^2$.
	In this model, the geodesics are precisely the intersections with $\mathbb{B}^d$
	of Euclidean circles and lines that meet the boundary
	$\partial\mathbb{B}^d=\mathbb{S}^{d-1}$ orthogonally.
	
	The hyperbolic distance between the origin and a point $y \in \mathbb{B}^d$
	is given explicitly by
	\[
	\mathrm{d}_{\mathrm{P}}(0,y)
	= 2 \operatorname{artanh}(\|y\|).
	\]
	More generally, for $x,y \in \mathbb{B}^d$ the
	hyperbolic distance satisfies
	\[
	\cosh\bigl( \mathrm{d}_{\mathrm{P}}(x,y) \bigr)
	= 1 + \frac{2\|x-y\|^2}{(1-\|x\|^2)(1-\|y\|^2)}.
	\]
	
	A totally geodesic hyperplane in $\mathbb{H}^d$ is a $(d-1)$-dimensional submanifold such that any geodesic contained in the submanifold is also a geodesic of the ambient space. In the Poincar\'e ball model, these hyperplanes are represented by the intersection of $\mathbb{B}^d$ with Euclidean $(d-1)$-spheres or Euclidean hyperplanes that intersect the boundary $\mathbb{S}^{d-1}$ orthogonally. Another important class of hypersurfaces consists of  horospheres. Intuitively, these can be viewed as the limits of hyperbolic spheres as their radii tend to infinity. In the Poincar\'e ball model, a horosphere is realized as a Euclidean sphere contained in the closure of $\mathbb{B}^d$ that is tangent to the boundary $\mathbb{S}^{d-1}$ at exactly one point (excluding the point of tangency itself).

	The Poincar\'e ball model for hyperbolic space will be our main
	tool for the geometric computations appearing later in the paper.

	\subsection{$\lambda$-geodesic hyperplanes}
	\label{sec:lambda-geodesic_hyperplanes}
	Let $0\leq\lambda\leq 1$. Following \cite{Solanes}, a
	$\lambda$--geodesic hyperplane is a complete totally umbilical hypersurface in $\mathbb{H}^d$ with normal curvature $\lambda $. In the Poincar\'e ball model, $\lambda$-geodesic hyperplanes are of the form $S\cap \mathbb{B}^d$, where $S$ is a $(d-1)$-dimensional sphere in $\mathbb{R}^d$ intersecting $\mathbb{S}^{d-1}$ at angle $\theta$ with $\cos(\theta)=\lambda$. In particular, $\lambda=0$ corresponds to the case of totally geodesic hyperbolic hyperplanes and $\lambda=1$ to the case of horospheres.
	Let $\mathrm{Hyp}_\lambda$ denote the space of $\lambda$-geodesic hyperplanes. For a fixed origin $o\in\mathbb{H}^d$, an element $H\in\mathrm{Hyp}_\lambda$ can be parametrised by a pair $(s,u)\in\mathbb{R}\times\mathbb{S}^{d-1}$, where $s\in\mathbb{R}$ is the signed distance from the origin $o$ to $H$ and $u$ is the unit vector in the tangent space of $\mathbb{H}^d$ at $o$ along the geodesic through $o$ intersecting $H$ orthogonally, pointing outside of the convex side. In the Poincar\'e ball model, if $o$ is the centre of $\mathbb{B}^d$, then $u\in\mathbb{S}^{d-1}$ is the direction which is orthogonal to $H$. The sign of $s$ is chosen in such a way that $s<0$ if $o$ lies on the convex side of $H$. The parametrised version of $H$ is denoted by $H_h(s,u)$.
	
	As argued in \cite{Solanes}, there exists an isometry-invariant measure $\nu_\lambda$ on $\mathrm{Hyp}_\lambda$, which is unique up to normalization. Following \cite{KRT22}, we choose the normalization in such a way that
	\begin{align}\label{eq:InvariantMeasure}
		\int_{\mathrm{Hyp}_\lambda} f(H)\; \nu_\lambda(\mathrm{d} H)=\int_{\mathbb{R}}\int_{\mathbb{S}^{d-1}}f(H_h(s,u))(\cosh(s)+\lambda\sinh(s))^{d-1}\;\sigma_{d-1}(\mathrm{d}u)\;\mathrm{d}s,
	\end{align}
	for all measurable functions $f\colon\mathrm{Hyp}_\lambda\to\mathbb{R}_{\geq 0}$,
	where $\sigma_{d-1}$ denotes the normalised surface measure on $\mathbb{S}^{d-1}$ and $\mathrm{d}s$ refers to the integration with respect to the Lebesgue measure on $\mathbb{R}$. Note that in contrast to \cite{KRT22}, we choose the sign of $s$ differently so that the invariant measure is of a slightly different form here.
	
	As we are not interested in all $\lambda$-geodesic hyperplanes, we further introduce the space $\mathrm{Hyp}_\lambda^o$ of all $\lambda$-geodesic hyperplanes for which $o$ does not lie on their convex side, i.e., for which $s>0$ in the previously mentioned parametrization. In this paper, $\mathrm{Hyp}_\lambda^o$ will be equipped with the measure $\nu_\lambda^o$ defined as the restriction of $\nu_\lambda$ to $\mathrm{Hyp}_\lambda^o$, i.e.,
	\begin{align*}
		\nu_\lambda^o(\,\cdot\,)=\int_0^\infty\int_{\mathbb{S}^{d-1}}{\bf 1}{\{H_h(s,u)\in\,\cdot\,\}}(\cosh(s)+\lambda\sinh(s))^{d-1}\;\sigma_{d-1}(\mathrm{d}u)\;\mathrm{d}s.
	\end{align*}

	\subsection{A Poisson process on the space of $\lambda$-geodesic hyperplanes}
	
	For $\gamma>0$ let $\eta_{\gamma,\lambda}$ be a Poisson process on $\mathrm{Hyp}_\lambda^o$ with intensity measure $\mu_{\gamma,\lambda}=\gamma\nu_\lambda^o$, see \cite[Chapter 3]{SW08} for the definition of Poisson processes on general state spaces. We want to represent this process in the Poincar\'e ball model. To pass to this model, we set $o$ to be the centre of the ball $\mathbb{B}^d$ and $r=\tanh\big(\frac{s}{2}\big)$ for $s\geq 0$. In other words, if a point has hyperbolic distance $s$ to $o$, then its Euclidean distance to $o$ is $r$. Accordingly, we denote a $\lambda$-geodesic hyperplane $H_h(s,u)$ as defined in the previous subsection by $H(r,u)$ in this model. This notation refers to the same geometric object, but the first coordinate is now the Euclidean distance $r \in [0,1)$ given by $r=\tanh(s/2)$.
	Then, we use that for $x\in\mathbb{R}$, $\sinh(2x) = 2 \sinh(x)\cosh(x)$ and $\cosh(2x) = 2\cosh^2(x) - 1$. Together with the relations
	\[
	\cosh(\atanh(x)) = \frac{1}{\sqrt{1 - x^2}} \qquad \text{and} \qquad\sinh(\atanh(x)) = \frac{x}{\sqrt{1 - x^2}}\quad
	\]
	for $x\in(0,1)$, we get
	\[
	\cosh(2  \atanh(r)) = 2 \cosh^2(\atanh(r)) - 1 = \frac{2}{1 - r^2} - 1 = \frac{1 + r^2}{1 - r^2}
	\]
	and
	\[
	\sinh(2 \atanh(r)) = 2 \sinh(\atanh(r)) \cosh(\atanh(r)) = \frac{2r}{1 - r^2}.
	\]
	for $r\in(0,1)$.
	Since $\mathrm{d}(2\atanh(r)) = \frac{2\;\mathrm{d} r}{1 - r^2} $, it follows that the intensity measure $\mu_{\gamma,\lambda}$ can be written as
	\begin{align}
		\mu_{\gamma,\lambda}(\,\cdot\,)&=\gamma\int_{0}^{1}\int_{\mathbb{S}^{d-1}}{\bf 1}{\{H(r,u) \in \,\cdot\,\}} \left( \frac{1 + r^2}{1 - r^2} +\lambda \cdot \frac{2r}{1 - r^2} \right)^{d-1}\cdot \frac{2}{1 - r^2}\;\sigma_{d-1}(\mathrm{d}u)\mathrm{d}r\nonumber\\
		&=\gamma\int_{0}^{1}\int_{\mathbb{S}^{d-1}}{\bf 1}{\{H(r,u) \in \,\cdot\,\}} \frac{2(1 + 2\lambda r + r^2)^{d-1}}{(1 - r^2)^d}\;\sigma_{d-1}(\mathrm{d}u)\mathrm{d}r. \label{eq:invariant_measure_poincare_disc_model}
	\end{align}

		\section{Visibility to infinity}\label{sec:Visibility}
		
	In the following we study the question whether, with positive probability, there exists an infinite geodesic ray emanating from $o$, which does not intersect any of the $\lambda$-geodesic hyperplanes of $\eta_{\gamma,\lambda}$. Note that, looking from $o$ in the direction of some $\lambda$-geodesic hyperplane, this $\lambda$-geodesic hyperplane covers a part of $\mathbb{S}^{d-1}$. Thus, the question raised above is equivalent to the question whether the shadows of the $\lambda$-geodesic hyperplanes of $\eta_{\gamma,\lambda}$ on $\mathbb{S}^{d-1}$, the boundary of the Poincar\'e ball $\mathbb{B}^d$, cover the whole sphere. In what follows, we choose $o$ to be the centre of $\mathbb{B}^d$.
	
	\subsection{Covering with spherical caps}\label{sec:Covering}
	
	 We start with analysing which part of $\mathbb{S}^{d-1}$ is covered by a fixed $\lambda$-geodesic hyperplane in the Poincar\'e ball model. Recall at first that every $\lambda$-geodesic hyperplane $H(r,u)$ in the Poincar\'e model is of the form $S\cap \mathbb{B}^d$, where $S$ is a $(d-1)$-dimensional sphere intersecting $\mathbb{S}^{d-1}$ at angle $\theta=\arccos(\lambda)$, $u\in\mathbb{S}^{d-1}$ is the unit vector intersecting $S$ orthogonally and $r$ is the Euclidean distance from $o$ to $S$. Thus, $H(r,u)$ covers a closed spherical cap. The corresponding open spherical cap, which we aim to analyze, is of the form
	 $$
	 \mathcal{S}^\circ(u,\varphi(r)):=\{x\in\mathbb{S}^{d-1}:\langle x, u\rangle>\varphi(r)\},
	 $$ where $\varphi:(0,1)\to(0,1)$ is a suitable function depending only on $r$ and $\lambda$. We start by making $\varphi$ explicit and refer to Figure \ref{fig: configuration_caps} for a geometric illustration in dimension $2$.
	 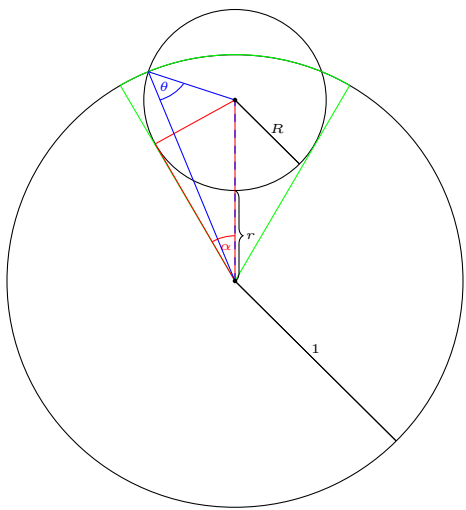
\begin{figure}
	 	\centering
	 		\begin{tikzpicture}[scale=3]
	 		
	 		\coordinate (O) at (0,0); 
	 		\coordinate (C) at (0,0.8); 
	 		\draw (O) circle(1); 
	 		\draw (C) circle(0.4); 
	 		
	 		\draw (O) -- ++(-45:1) coordinate (P);
	 		\draw (O) -- (P);
	 		\node[above] at ($(O)!0.5!(P)$) {\tiny$1$};
	 		
	 		\draw[green] (O) -- ++(120.2:1) coordinate (P);
	 		\draw[green] (O) -- ++(59.8:1) coordinate (P);
	 		\draw[green] (0.503,0.865) arc[start angle=59.8, end angle=120.2, radius=1cm] ;

	 		\draw (C) -- ++(-45:0.4) coordinate (Rpt);
	 		\draw (C) -- (Rpt);
	 		\node[above] at ($(C)!0.65!(Rpt)$) {\tiny$R$};
	 		
	 		\draw[red] (0,0)--(0,0.8);
	 		\draw[red] (-0.35,0.606)--(0,0.8);
	 		\draw[red] (0,0)--(-0.35,0.606);
	 		\node[above, red] at (-0.04,0.1) {\tiny$\alpha$};
	 		\draw[red] (0,0.2) arc[start angle=90, end angle=120.2, radius=0.2cm] ;
	 		
	 		\draw[blue,dashed] (0,0)--(0,0.8);
	 		\draw[blue] (-0.38,0.925)--(0,0.8);
	 		\draw[blue] (0,0)--(-0.38,0.925);
	 		\node[below, blue] at (-0.31,0.92) {\tiny$\theta$};
	 		\draw[blue] (-0.33,0.8) arc[start angle=-74, end angle=-37, radius=0.2cm] ;
	 		
	 		\draw[decorate,decoration={brace,amplitude=3pt},xshift=0pt,yshift=0pt] (0,0.4) -- (0,0);
	 		\node[right] at (0.01,0.2) {\tiny$r$};
	 		
	 		\fill (O) circle(0.01);
	 		\fill (C) circle(0.01);
	 		
	 	\end{tikzpicture}
	 	\caption{We consider a $\lambda$-geodesic hyperplane of Euclidean distance $r$ to $o$ realized by the intersection of $\mathbb{B}^2$ with the black sphere of radius $R$, where $R$ depends on $r$ and $\lambda=\cos(\theta)$. This $\lambda$-geodesic hyperplane covers a part of $\mathbb{S}^1$ illustrated in green. In order to describe this green cap we aim to determine $\varphi(r)=\cos(\alpha)$. To this end, we use the red and the blue triangle. The red triangle provides $\sin(\alpha)=\frac{R}{R+r}$. Applying the law of cosine to the blue triangle leads to $2R\cos(\theta)=R^2+1-(R+r)^2$. }
	 	\label{fig: configuration_caps}
	 \end{figure}
	
	 Let $R$ denote the radius of $S$. Then,
	 $$
	 \varphi(r)=\cos\Big(\arcsin\Big(\frac{R}{R+r}\Big)\Big),
	 $$
	 compare with the red triangle in Figure \ref{fig: configuration_caps}. Hence, it remains to express $R$ in terms of $r$ and $\theta$. Using the law of cosine for a triangle that has $0$, the centre of $S$ and an arbitrary intersection point of $S$ and $\mathbb{B}^d$ as vertices leads to $\lambda=\cos(\theta)=\frac{R^2+1-(R+r)^2}{2R}$, see the blue triangle in Figure \ref{fig: configuration_caps}. Thus, $R=\frac{1-r^2}{2(\lambda+r)}$.
	Altogether, we conclude that
	\begin{align}
		\varphi(r)&=\cos\Bigg(\arcsin\Bigg(\frac{\frac{1-r^2}{2(\lambda+r)}}{\frac{1-r^2}{2(\lambda+r)}+r}\Bigg)\Bigg)=\cos\Big(\arcsin\Big(\frac{1-r^2}{1+r^2+2r\lambda}\Big)\Big)\nonumber\\&=\sqrt{1-\Big(\frac{1-r^2}{1+r^2+2r\lambda}\Big)^2}=\frac{2\sqrt{r(\lambda+r)(1+\lambda r)}}{1+2\lambda r+r^2}.\label{eq:varphi(r)}
	\end{align}
	
	In the following lemma we derive the asymptotic behaviour of $\varphi(r)$ as $r\uparrow 1$.
	
	\begin{lemma}\label{lemma:asymptotics_varphi}
		Let $\varphi:(0,1)\to(0,1)$ be defined as in \eqref{eq:varphi(r)}. Then, as $r\uparrow 1$,
		\begin{align*}
			1-\varphi(r)
			&=\frac{(1-r)^2}{2(1+\lambda)^2}
			+\frac{(1-r)^3}{2(1+\lambda)^2}
			+o\bigl((1-r)^3\bigr).
		\end{align*}
	\end{lemma}
	\begin{proof}
		Let $f\colon[0,1]\to[0,1]$ and $g\colon[0,1]\to[0,1]$ be given by
		\[
		f(r):=\frac{1-r^2}{1+r^2+2r\lambda},
		\qquad
		g(r):=1-\sqrt{1-r^2},
		\]
		and set $h:=g\circ f=1-\varphi$.
		Write $r=1-\varepsilon$ with $\varepsilon\downarrow 0$. Then
		\begin{align*}
			f(1-\varepsilon)
			&=\frac{2\varepsilon-\varepsilon^2}{2(1+\lambda)-(2+2\lambda)\varepsilon+\varepsilon^2}
			=\frac{\varepsilon}{1+\lambda}+\frac{\varepsilon^2}{2(1+\lambda)}+O(\varepsilon^3).
		\end{align*}
		Further, using the Taylor expansion of $\sqrt{1-x}$ at $x=0$ and substituting $x=y^2$, we obtain
		\begin{align*}
			g(y)
			=1-\sqrt{1-y^2}
			=\frac{y^2}{2}+O(y^4),
			\qquad y\downarrow 0.
		\end{align*}
		Combining the last two displays yields
		\begin{align*}
			1-\varphi(1-\varepsilon)
			=h(1-\varepsilon)
			=g\bigl(f(1-\varepsilon)\bigr)
			&=\frac12\,f(1-\varepsilon)^2+O\bigl(f(1-\varepsilon)^4\bigr)\\
			&=\frac12\left(\frac{\varepsilon}{1+\lambda}+\frac{\varepsilon^2}{2(1+\lambda)}+O(\varepsilon^3)\right)^2+O(\varepsilon^4)\\
			&=\frac{\varepsilon^2}{2(1+\lambda)^2}+\frac{\varepsilon^3}{2(1+\lambda)^2}+o(\varepsilon^3).
		\end{align*}
		Since $\varepsilon=1-r$, this proves the claim.
	\end{proof}

%
	
	\subsection{Criterion for visibility}
	
	For $u\in\mathbb{S}^{d-1}$ and $h\geq 0$ let $\mathcal{S}^\circ(u,h)=\{x\in\mathbb{S}^{d-1}:\langle x, u\rangle>h\}$ be the open spherical cap centred at $u$ with height $h$ and denote by $\mathcal{S}(u,h)$ its closure.
	The following theorem is taken from \cite[Theorem 4.1]{GKT22}. It is a reformulation of a result of Hoffmann-Jørgensen, see \cite[Formulas (3.10), (5.3) and (5.5)]{H73}.
	
	\begin{prop}\label{thm: Hoffmann_Jorgensen}
		Let $u_1,u_2,\ldots$ be independent and uniformly distributed on $\mathbb{S}^{d-1}$. Let the sequence $h_1, h_2, \ldots$ of real numbers in $(0,1)$ be such that
		\begin{align*}
		\lim_{n \to \infty} n \sigma_{d-1}(\mathcal{S}^\circ(u_n, h_n)) = a \in [0, \infty].
	\end{align*}
		\begin{itemize}
			\item[(a)] If $a > 1$, then $\displaystyle\mathbb{P}[\limsup_{n\to\infty} \mathcal{S}^\circ(u_n, h_n) = \mathbb{S}^{d-1}] = 1$, i.e., with probability one, each point of the sphere is covered infinitely often by the caps.
			\item[(b)] If $a < 1$, $\displaystyle\mathbb{P}[\limsup_{n\to\infty} \mathcal{S}^\circ(u_n, h_n) = \mathbb{S}^{d-1}] = 0$ and $\displaystyle\mathbb{P}\big[\bigcup_{n=1}^\infty \mathcal{S}(u_n, h_n) \neq \mathbb{S}^{d-1}\big] > 0$. In particular, the probability that there exist points on the sphere that are not covered by the closed caps is positive.
		\end{itemize}
	\end{prop}

	Recall that we are interested in the behaviour of $\mathcal{S}^\circ(u_n,\varphi(r_n))$, where, by \eqref{eq:invariant_measure_poincare_disc_model}, $u_1,u_2,\dots$ are uniformly distributed on $\mathbb{S}^{d-1}$ and $0<r_1<r_2<\ldots$ form an inhomogeneous Poisson point process on $(0,1)$ with intensity function
	\begin{align}\label{eq:intensity_r}
		f(r) := \frac{2\gamma(1 +2\lambda r + r^2)^{d-1}}{(1 - r^2)^d},\qquad 0<r<1.
	\end{align}
	We note that $\int_0^1 f(r)\,\dint r=+\infty$. In what follows, we use the notation $f(x)\sim g(x)$ as $x\to x_0$ for two functions $f$ and $g$ and some $x_0\in\mathbb{R}$ if $\lim\limits_{x\to x_0}\frac{f(x)}{g(x)}=1$. Moreover, for a sequence of random variables $(X_n)_{n\geq 1}$ and a positive deterministic sequence $(a_n)_{n\geq 1}$, we write $X_n=O(a_n)$ almost surely if, with probability one, there exist a constant $C>0$ and an integer $N\geq 1$ such that $|X_n|\leq C a_n$ for all $n\geq N$.
	
	\begin{lemma}\label{lemma:asymptotics_rn}
		Almost surely, it holds that
		\begin{align*}
			1-r_n\sim \Big(\frac{\gamma}{(d-1)n}\Big)^{1/(d-1)}(1+\lambda)
		\end{align*}
		as $n\to\infty$. Moreover, in the case $d=2$ we have almost surely
		\begin{align}\label{eq:rn_d2_second_order}
			1-r_n
			=\frac{\gamma(1+\lambda)}{n}
				+O\left(\sqrt{\frac{\log\log(n)}{n^3}}\right),
		\end{align}
	as $n\to\infty$.
	\end{lemma}
	\begin{proof} The first part of the following proof is analogous to the proof of \cite[Lemma 4.2]{GKT22} for $\lambda=0$.
		 Let $P_1<P_2<\ldots$ be arrivals of a homogeneous Poisson point process on $(0,\infty)$ with intensity $1$. We use the distributional representation $\tau(P_n)=r_n$, where $\tau:(0,\infty)\to (0,1)$ is monotone increasing with
		\begin{align*}
			\int_0^{\tau(y)}\frac{2\gamma(1 +2\lambda r + r^2)^{d-1}}{(1 - r^2)^d}\;\mathrm{d}r=y
		\end{align*}
		for all $y\in(0,\infty)$.
		By L'Hospitals rule we have
		\begin{align*}
			\lim\limits_{z\uparrow 1}\frac{	\int_0^{z}\frac{2\gamma(1 +2\lambda r + r^2)^{d-1}}{(1 - r^2)^d}\;\mathrm{d}r}{\frac{\gamma(1+\lambda)^{d-1}}{(d-1)}(1-z)^{1-d}}=	\lim\limits_{z\uparrow 1}\frac{\frac{2\gamma(1 +2\lambda z + z^2)^{d-1}}{(1 - z^2)^d}}{\gamma(1+\lambda)^{d-1}(1-z)^{-d}}=\lim\limits_{z\uparrow 1}\frac{2(1 +2\lambda z + z^2)^{d-1}}{(1+z)^d(1+\lambda)^{d-1}}=1.
		\end{align*}
		Thus,
		\begin{align*}
				\int_{0}^{z}\frac{2\gamma(1 +2\lambda r + r^2)^{d-1}}{(1 - r^2)^d}\;\mathrm{d}r\sim \frac{\gamma(1+\lambda)^{d-1}}{(d-1)}(1-z)^{1-d}
		\end{align*}
		as $z\uparrow 1$. Since $\tau(y)\to 1$ as $y\to \infty$ we have
		\begin{align*}
			\frac{\gamma(1+\lambda)^{d-1}}{(d-1)}(1-\tau(y))^{1-d}\sim y
		\end{align*}
		as $y\to\infty$. Since $P_n \sim n$ almost surely as $n\to\infty$ and $\tau(P_n)=r_n$, it
		follows that almost surely
		\[
		1-r_n = 1-\tau(P_n)
		\sim \Big(\frac{\gamma}{(d-1)P_n}\Big)^{1/(d-1)}(1+\lambda)
		\sim \Big(\frac{\gamma}{(d-1)n}\Big)^{1/(d-1)}(1+\lambda),
		\]
		which proves the first claim.
		
		Assume now that $d=2$ and set
		\[
		F(z):=\int_0^{z}\frac{2\gamma(1+2\lambda r+r^2)}{(1-r^2)^2}\,\dint r,
		\qquad z\in(0,1).
		\]
		A direct computation shows that
		\begin{equation}\label{eq:F_explicit_d2_inlemma}
			F(z)=\frac{2\gamma z(1+\lambda z)}{1-z^2},\qquad z\in(0,1).
		\end{equation}
		Indeed, differentiating the right-hand side yields the integrand and $F(0)=0$. We can now argue as before, but in view of the explicit form of $F(z)$ in \eqref{eq:F_explicit_d2_inlemma} the error terms can be controlled more precisely. Namely, by definition of $\tau$, we have $F(\tau(y))=y$ for all $y>0$.
		Write $\varepsilon:=1-z$ and expand \eqref{eq:F_explicit_d2_inlemma} at $z=1$ to obtain, as $\varepsilon\downarrow 0$,
		\begin{equation}\label{eq:F_expansion_d2_inlemma}
			F(1-\varepsilon)
			=\frac{\gamma(1+\lambda)}{\varepsilon}-\frac{\gamma}{2}(1+3\lambda)+O(\varepsilon).
		\end{equation}
		Since $\tau(y)\uparrow 1$ as $y\to\infty$, setting $\varepsilon(y):=1-\tau(y)$ and inserting $z=\tau(y)$ into
		\eqref{eq:F_expansion_d2_inlemma} yields
		\[
		y=\frac{\gamma(1+\lambda)}{\varepsilon(y)}-\frac{\gamma}{2}(1+3\lambda)+O(\varepsilon(y)),
		\qquad y\to\infty.
		\]
		Inverting this relation gives
		\begin{equation}\label{eq:eps_asymptotics_d2_inlemma}
			1-\tau(y)
			=\frac{\gamma(1+\lambda)}{y}
			-\frac{\gamma^2}{2}\frac{(1+\lambda)(1+3\lambda)}{y^2}
			+o\Big(\frac{1}{y^2}\Big),
			\qquad y\to\infty.
		\end{equation}
		Moreover, by the law of iterated logarithm it holds $P_n=n+O(\sqrt{n\log\log(n)})$ almost surely as $n\to\infty$. Thus, using Taylor expansion of the function $g(x)=\frac{1}{x}$, we get almost surely,
		\begin{align*}
			\frac{1}{P_n}=\frac{1}{n}+O\left(\frac{\sqrt{n\log\log(n)}}{n^2}\right)=\frac{1}{n}+O\left(\sqrt{\frac{\log\log(n)}{n^3}}\right).
		\end{align*}
		Finally, since $r_n=\tau(P_n)$, substituting $y=P_n$ into
		\eqref{eq:eps_asymptotics_d2_inlemma} yields almost surely
		\begin{align*}
		1-r_n
		=1-\tau(P_n)
		&=\frac{\gamma(1+\lambda)}{P_n}
		-\frac{\gamma^2}{2}\frac{(1+\lambda)(1+3\lambda)}{P_n^2}
		+o\Big(\frac{1}{P_n^2}\Big)\\
		&=\frac{\gamma(1+\lambda)}{n}
		+O\left(\sqrt{\frac{\log\log(n)}{n^3}}\right),
		\end{align*}
		which proves \eqref{eq:rn_d2_second_order}.
	\end{proof}

	We are now prepared to derive the following result, which in particular covers part (a) of Theorem \ref{thm:MainIntro} stated in the introduction.

	\begin{theorem}\label{thm:CoveringGenerald}
		Let $u_1,u_2,\dots$ be independent and uniformly distributed on $\mathbb{S}^{d-1}$, let $0<r_1<r_2<\ldots$ be such that they form an inhomogeneous Poisson point process on $(0,1)$ with intensity function $f$, where $f$ is defined as in \eqref{eq:intensity_r}, and assume that the sequences $(u_n)_{n\geq 1}$ and $(r_n)_{n\geq 1}$ are independent.
		Let $\gamma_\mathrm{crit}$ be the constant in \eqref{eq:GammaCrit}.
		\begin{itemize}
			\item[a)] If $\gamma>\gamma_\mathrm{crit}$, then $\displaystyle\mathbb{P}[\limsup_{n\to\infty} \mathcal{S}^\circ(u_n, \varphi(r_n)) = \mathbb{S}^{d-1}] = 1$. Thus, the visibility set $Z_{\gamma,\lambda,d}$ in \eqref{eq:DefVisibilitySet} is bounded almost surely.
			\item[b)] If $\gamma<\gamma_\mathrm{crit}$, then $\displaystyle\mathbb{P}[\limsup_{n\to\infty} \mathcal{S}^\circ(u_n, \varphi(r_n)) = \mathbb{S}^{d-1}] = 0$ and $\displaystyle\mathbb{P}\big[\bigcup_{n=1}^\infty \mathcal{S}(u_n, \varphi(r_n)) \neq \mathbb{S}^{d-1}\big] > 0$. In particular, the probability that the visibility set $Z_{\gamma,\lambda,d}$ is unbounded is strictly positive.
		\end{itemize}
	\end{theorem}
	\begin{proof}
		Let $u\in\mathbb{S}^{d-1}$ and $0\leq h\leq 1$. Then,
		\begin{equation*}
			\sigma_{d-1}(\mathcal{S}^\circ(u,h)) =  \frac{\Gamma(\frac{d}{2})}{\sqrt{\pi}\,\Gamma(\frac{d-1}{2})}\int_h^1(1-s^2)^{d-3\over 2}\,\dint s
		\end{equation*}
		according to \cite[Equation (2.7)]{KSTBook}. It follows that,
		\begin{align*}
			\sigma_{d-1}(\mathcal{S}^\circ(u,h))\sim c_d(1-h)^{\frac{d-1}{2}},\qquad h\uparrow 1,
		\end{align*}
		for
		$$
		c_d := \frac{\Gamma(\frac{d}{2})}{\sqrt{\pi}\,\Gamma(\frac{d-1}{2})}\cdot {2^{d-1\over 2}\over d-1} = \frac{2^{\frac{d-1}{2}}\Gamma(\frac{d}{2})}{\sqrt{\pi}(d-1)\Gamma(\frac{d-1}{2})}.
		$$ 		
		Since almost surely $r_n\to 1$ as $n\to\infty$, combining Lemma \ref{lemma:asymptotics_varphi} and Lemma \ref{lemma:asymptotics_rn} leads to
		\begin{align}\label{eq: 1-phi(rn)_asymptotics}
			1-\varphi(r_n)\sim \frac{(1-r_n)^2}{2(1+\lambda)^2}\sim{1\over 2}\Big({\gamma\over (d-1)n}\Big)^{2\over d-1}
		\end{align}
		as $n\to\infty $. It follows that almost surely
		$$
		\sigma_{d-1}(\mathcal{S}^\circ(u_n,\varphi(r_n)))\sim \frac{c_d\gamma}{2^{(d-1)/2}(d-1)n}=\frac{\gamma\Gamma(\frac{d}{2})}{\sqrt{\pi}(d-1)^2\Gamma(\frac{d-1}{2})n}
		$$
		as $n\to\infty$. Applying Proposition \ref{thm: Hoffmann_Jorgensen} completes the proof with $\gamma_\mathrm{crit}$ as in \eqref{eq:GammaCrit}. The translation to the visibility set follows from the construction described in Section \ref{sec:Covering}.
	\end{proof}

\begin{rema}\rm
It is natural to ask whether the visibility event
 '\textit{there exists a geodesic ray from $o$ avoiding all $\lambda$--geodesic hyperplanes of $\eta_{\gamma,\lambda}$}'
is monotone in $\lambda$ in a suitable coupling.  On an abstract level, if two locally finite intensity measures on the same state space satisfy $\nu_1\le \nu_2$, then one can construct Poisson point processes $\Pi_1$ and $\Pi_2$ with intensity measures $\nu_1$ and $\nu_2$, respectively, on a common probability space such that $\Pi_1\subseteq \Pi_2$ almost surely.
	In our setting, this idea can be applied to the induced Poisson process on
	$\mathbb S^{d-1}\times(0,1)$ that generates the random family of shadows
	$\{\mathcal S(u_n,\varphi(r_n))\}_{n\ge1}$.
	Indeed, using~\eqref{eq:varphi(r)} and~\eqref{eq:intensity_r} one verifies that the intensity of the Poisson point process $(\varphi(r_n))_{n\ge1}$ on $(0,1)$ is given by
$$
g(s;\lambda) =\frac{\gamma\, s}   {(1-s^2)^{\frac{d+1}{2}}\, \sqrt{s^2+\lambda^2(1-s^2)}}, \qquad  0<s<1.
$$
For fixed $s$, this function is decreasing in $\lambda$, which means that there exists a coupling of the shadow processes on $\mathbb S^{d-1}$ such that, for all $0\leq \lambda_1 < \lambda_2\leq 1$, every shadow present in the process with parameter  $\lambda_2$ is also present in the process with parameter $\lambda_1$.
However, this does not  yield a coupling of the underlying $\lambda$-hyperplane processes $\{\eta_{\gamma,\lambda}\}_{\lambda\in[0,1]}$ on the state spaces $\mathrm{Hyp}_\lambda^o$.

Another  natural attempt to construct a coupling of the $\lambda$-hyperplane processes is to use that,  for $\lambda\in(0,1)$,  a $\lambda$--geodesic hyperplane is an equidistant hypersurface to a totally geodesic hyperplane. However,  passing to an equidistant may change on which side of the hypersurface the origin lies. Hence a naive attempt to couple across $\lambda$ by mapping each $H\in\mathrm{Hyp}_0^o$ to one of its equidistants may leave the state space $\mathrm{Hyp}_\lambda^o$, and any subsequent correction (for instance by discarding such hypersurfaces) would alter the intensity in a $\lambda$--dependent way.
\end{rema}


\subsection{The case $\gamma=\gamma_{\mathrm{crit}}$}

A natural question is what happens at the critical value $\gamma=\gamma_{\mathrm{crit}}$.
In continuum percolation models, the behaviour at criticality is typically the most delicate part of the phase transition and has been the focus of substantial recent work.
In particular, for the isometry-invariant Poisson hyperplane tessellations in $\mathbb{H}^d$ it was recently shown that at the critical intensity there are no unbounded cells, thus settling the critical case in all dimensions, see \cite[Theorem 1.1]{BuehlerGusakovaRecke}.
This naturally raises the analogous question in our setting: whether, for $\lambda\in[0,1]$, the visibility set $Z_{\gamma,\lambda,d}$ is almost surely bounded at $\gamma=\gamma_{\mathrm{crit}}$.
Our argument based on Proposition~\ref{thm: Hoffmann_Jorgensen} does not apply at the boundary value $a=1$, and the strategy of \cite[Section~3]{BuehlerGusakovaRecke} relies on working with an isometry-invariant intensity measure.
Since we restrict to $\mathrm{Hyp}_\lambda^o$, the measure $\mu_{\gamma,\lambda}$ is not isometry-invariant, and the critical-case arguments from the isometry-invariant model do not transfer directly.
For this reason, we do not address the critical case in full generality here.
Nevertheless, using a criterion from \cite{Shepp}, we can show that in dimension $d=2$ the visibility region $Z_{\gamma,\lambda,2}$ is almost surely bounded for every $\lambda\in[0,1]$ even at $\gamma=\gamma_{\mathrm{crit}}$, proving the last claim in Theorem \ref{thm:MainIntro} (a).

	\begin{prop}\label{prop:gamma=gamma_crit}
		We assume the same set-up as in Theorem \ref{thm:CoveringGenerald}. For $d=2$ and $\gamma=\gamma_{\mathrm{crit}}$ it holds
		$\displaystyle\mathbb{P}[\limsup_{n\to\infty} \mathcal{S}^\circ(u_n, \varphi(r_n)) = \mathbb{S}^{1}] = 1$. Thus, the visibility set $Z_{\gamma_\mathrm{crit},\lambda,2}$ in \eqref{eq:DefVisibilitySet} is almost surely bounded.
	\end{prop}
	\begin{proof}
		As before, let  $0<r_1<r_2<\dots$ be such that they form an inhomogeneous Poisson point process on $(0,1)$ with intensity function $f$ defined in \eqref{eq:intensity_r}. Then, in the Poincar\'e disc model a $\lambda$-geodesic hyperplane $H(r_n,u)$ of distance $r_n$ to $o$ in direction $u\in\mathbb{S}^1$ covers an arc of length $2\arccos(\varphi(r_n))$ for $n\in\mathbb{N}$ (see Figure \ref{fig: configuration_caps}). Now, by \cite[Equation (1)]{Shepp}, Proposition \ref{prop:gamma=gamma_crit} follows once we have shown that
		$$\sum_{n=1}^{\infty}\frac{1}{n^2}e^{\ell_1+\ldots+\ell_n}=\infty,$$
		where $2\pi\ell_i$ is the arc length covered by $r_i$, i.e.\ $\ell_i=\frac{1}{\pi}\arccos(\varphi(r_i))$ for $i\in\mathbb{N}$. Combining \eqref{eq: 1-phi(rn)_asymptotics}, $\gamma_{\mathrm{crit}}=\pi$ and the refined $d=2$ asymptotics from Lemma~\ref{lemma:asymptotics_rn} we can sharpen the estimate for the arc lengths $\ell_i$.
		Indeed, for $d=2$ and $\gamma=\gamma_{\mathrm{crit}}=\pi$, Lemma~\ref{lemma:asymptotics_rn} yields almost surely
		\begin{equation*}
			1-r_i=\frac{\pi(1+\lambda)}{i} +O\left(\sqrt{\frac{\log\log(i)}{i^3}}\right),\qquad i\to\infty.
		\end{equation*}
		Together with Lemma~\ref{lemma:asymptotics_varphi} this implies almost surely
		\begin{equation*}
			\delta_i:=1-\varphi(r_i)
			=\frac{(1-r_i)^2}{2(1+\lambda)^2}+O\big((1-r_i)^3\big)
			=\frac{\pi^2}{2i^2}+O\left(\sqrt{\frac{\log\log(i)}{i^5}}\right),
			\qquad i\to\infty.
		\end{equation*}
		Using the Taylor expansions $\arccos(1-x)=\sqrt{2x}+O(x^{3/2})$ as $x\downarrow 0$ and $\sqrt{1+x}=1+\frac{x}{2}+O(x^2)$ as $x\downarrow 0$, we obtain almost surely
		\[
		\arccos(\varphi(r_i))
		=\arccos(1-\delta_i)
		=\sqrt{2\delta_i}+O(\delta_i^{3/2})
		=\frac{\pi}{i}+O\left(\sqrt{\frac{\log\log(i)}{i^3}}\right),
		\qquad i\to\infty.
		\]
		Consequently,
		\begin{equation*}
			\ell_i=\frac{1}{\pi}\arccos(\varphi(r_i))
			=\frac{1}{i}+O\left(\sqrt{\frac{\log\log(i)}{i^3}}\right),
			\qquad i\to\infty,
		\end{equation*}
		almost surely. Hence, as $\sum_{k=1}^\infty \sqrt{\frac{\log\log(k)}{k^3}}<\infty,$ this
		 implies that there exists $c\in\mathbb{R}$ such that
		\[
		\sum_{k=1}^{n}\ell_k
		\ge \sum_{k=1}^{n}\frac{1}{k}+c
		\ge \ln(n)+c,
		\]
		where we used $\sum_{k=1}^n \frac1k \ge \ln(n)$.
		Thus,
		\begin{align*}
			\sum_{n=1}^{\infty}\frac{1}{n^2}e^{\ell_1+\ldots+\ell_n}
			&\ge \sum_{n=1}^{\infty}\frac{1}{n^2}e^{\ln(n)+c}
			= \sum_{n=1}^\infty \frac{e^{c}}{n}
			=\infty,
		\end{align*}
		which completes the proof.
	\end{proof}

	\section{Measure of $\lambda$-geodesic hyperplanes hitting a segment}\label{sec:IntegralGeometrySegment}

\subsection{Integral representation}

If we wish to compute the expected volume of the set of points in $\mathbb{H}^d$ that are visible from $o$ in the presence of a Poisson process of $\lambda$-geodesic hyperplanes in Section \ref{sec:ExpectedVolume}, we are naturally led to a purely geometric problem: What is the measure $\nu_\lambda^o$ of $\lambda$-geodesic hyperplanes intersecting a geodesic segment of a given length starting in $o$?
When $\lambda = 0$, each totally geodesic hyperplane can intersect a geodesic segment at most once, and the corresponding measure is a linear function of the segment length, see \cite{BuehlerHugThaele}.
However, for every $\lambda > 0$, a $\lambda$-geodesic hyperplane may intersect a geodesic segment in $0$, $1$, or $2$ points.  In particular, a transition from an indicator function to the Euler characteristic as in \eqref{eq:IndicatorToEulerChar} is no more possible. Showing that the measure remains a linear function of the segment length in this setting is a nontrivial problem, and resolving it is the aim of this section.

To tackle this, we first need to characterize exactly which $\lambda$-geodesic hyperplanes $H(r,u)$ intersect a reference geodesic segment. It suffices to consider a segment of hyperbolic length $h$ starting at the origin and aligned with the first coordinate vector $e_1:=(1,0,\ldots,0)\in\mathbb{R}^d$. In the Poincar\'e ball model, this corresponds to the Euclidean segment connecting the origin to the point $\tanh(h/2)e_1$. The following lemma establishes the precise analytic conditions on the parameters $r$ and $u$ that are necessary and sufficient for such an intersection to occur.

\begin{lemma}\label{lemma:schnitt_indicator} Let $h>0$, $0\leq r<1$, $u\in\mathbb{S}^{d-1}$ and $u_1:=\langle u,e_1\rangle$. Then, it holds that $H(r,u)\cap\left[0, \tanh(\frac{h}{2})e_1\right]\neq\emptyset$ if and only if one of the following cases occurs:
	\begin{enumerate}
		\item $r=0$,
		\item $\lambda=0$, $r\in(0,\tah]$ and $u_1\geq \frac{\tahq r+r}{\tah(r^2+1)}$,
		\item  $\lambda\in(0,1]$, $r\in\left(0,r_c\right)\textit{ and } u_1\geq \frac{2\sqrt{(r^2\lambda+r)(\lambda+r)}}{r^2+2\lambda r+1}$,
		\item $\lambda\in(0,1]$, $r\in\left[r_c,\tah\right]\textit{ and }u_1\geq \frac{\tahq(\lambda+r)+r^2\lambda+r}{\tah(r^2+2\lambda r+1)}$,
	\end{enumerate}
	where $r_c:=\frac{\tahq-1}{2\lambda}+\sqrt{\frac{\left(1-\tahq\right)^2}{4\lambda^2}+\tahq}$.
\end{lemma}
\begin{proof}
	Note at first that in the Poincar\'e ball model for $\mathbb{H}^d$, $H(r,u)$ corresponds to the intersection of $\mathbb{B}^d$ with a $(d-1)$-dimensional sphere of radius $R$ with $R=\frac{1-r^2}{2(\lambda+r)}$ and centre $(r+R)u$, see Section \ref{sec:Covering} and Figure \ref{fig: configuration_caps}. Thus, $H(r,u)\cap\left[0, \tanh(\frac{h}{2})e_1\right]\neq\emptyset$ if there exists $x\in \left[0, \tanh(\frac{h}{2})\right]$ such that
	\begin{align*}
		R^2=\lVert (r+R)u-xe_1\rVert^2=(r+R)^2-2(r+R)xu_1+x^2,
	\end{align*}
	where $u_1=\langle u,e_1\rangle$. Hence, $H(r,u)\cap\left[0, \tanh(\frac{h}{2})e_1\right]\neq\emptyset$ if and only if the equation $$p(x):=x^2-2(r+R)xu_1+(r+R)^2-R^2=0$$ has a solution in $\left[0, \tanh(\frac{h}{2})\right]$.
	It holds that $p(x)=0$ if and only if
	$$
	x_{1,2}=(r+R)u_1\pm\sqrt{(r+R)^2u_1^2-(r+R)^2+R^2}.
	$$
	If $r=0$, $x=0$ is a solution of $p(x)=0$. Thus, $H(0,u)\cap\left[0, \tanh(\frac{h}{2})e_1\right]\neq\emptyset$.
	
	For $r>0 $ note at first that if real solutions exist, it holds $x_{1,2}> 0$ if $u_1> 0$ and $x_{1,2}< 0$ if $u_1<0$. Thus, $p(x)=0$ has a solution in $\left[0, \tanh(\frac{h}{2})\right]$ if and only if $u_1>0$ and
	\begin{align}
		(r+R)^2u_1^2-(r+R)^2+R^2&\geq 0\label{(C1)}\\
		(r+R)u_1-\sqrt{(r+R)^2u_1^2-(r+R)^2+R^2}&\leq\tah\label{(C2)}
	\end{align}
	are fulfilled.
	Condition \eqref{(C1)} is equivalent to
	\begin{align*}
		(r+R)^2u_1^2&\geq r^2+2Rr,\nonumber\\
		\iff  u_1&\geq\frac{\sqrt{r^2+2Rr}}{r+R}=:f_1(r)
	\end{align*}
	Condition \eqref{(C2)} is fulfilled if either
	\begin{align*}
		&(r+R)u_1-\tah\leq0\nonumber\\
		&\iff u_1\leq \frac{\tah}{r+R}=:f_2(r)
	\end{align*}
	or $u_1\geq \frac{\tah}{r+R}$ and
	\begin{align*}
		&\left((r+R)u_1-\tah\right)^2\leq (r+R)^2u_1^2-r^2-2Rr\nonumber\\
		&\iff -2(r+R)u_1\tah+\tahq\leq-r^2-2Rr\nonumber\\
		&\iff u_1\geq \frac{\tahq+r^2+2Rr}{2(r+R)\tah}=:f_3(r).
	\end{align*}
	Since $2ab\leq a^2+b^2$ for $a,b\in\mathbb{R}$, we have
	$$f_1(r)=\frac{2\tah\sqrt{r^2+2Rr}}{2\tah(r+R)}\leq \frac{\tahq+r^2+2Rr}{2(r+R)\tah}=f_3(r).$$
	
	Altogether, in case that $r^2+2Rr<\tahq$, it holds that
	\begin{align*}
		f_1(r)\leq f_3(r)\leq \frac{2\tahq}{2(r+R)\tah}=f_2(r).
	\end{align*}
	Thus, in this case, \eqref{(C2)} is automatically fulfilled and $H(r,u)\cap\left[0, \tanh(\frac{h}{2})e_1\right]\neq\emptyset$ if $u_1\geq f_1(r).$
	In the case that $r^2+2Rr\geq \tahq$, we find
	\begin{align*}
		f_2(r)\leq \frac{\sqrt{r^2+2Rr}}{r+R}=f_1(r)\leq f_3(r).
	\end{align*}
	Hence, in this case \eqref{(C1)} and \eqref{(C2)} are fulfilled, i.e., $H(r,u)\cap\left[0, \tanh(\frac{h}{2})e_1\right]\neq\emptyset$, if $u_1\geq f_3(r).$ Inserting $R=\frac{1-r^2}{2(\lambda+r)}$ leads to
	\begin{align*}
		&r^2+2Rr=r^2+\frac{2r(1-r^2)}{2(\lambda+r)}=\frac{r^2\lambda+r}{\lambda+r},\\
		&r+R=r+\frac{1-r^2}{2(\lambda+r)}=\frac{r^2+2\lambda r+1}{2(\lambda+r)},\\
		&f_1(r)=\sqrt{\frac{r^2\lambda+r}{\lambda+r}}\cdot\frac{2(\lambda+r)}{r^2+2\lambda r+1}=\frac{2\sqrt{(r^2\lambda+r)(\lambda+r)}}{r^2+2\lambda r+1},\\
		&f_3(r)=\frac{\tahq(\lambda+r)+r^2\lambda+r}{(\lambda+r)}\cdot\frac{(\lambda+r)}{\tah(r^2+2\lambda r+1)}=\frac{\tahq(\lambda+r)+r^2\lambda+r}{\tah(r^2+2\lambda r+1)}.
	\end{align*}
	Then, for $\lambda\in(0,1]$, the condition $r^2+2Rr<\tah$ can be rewritten as follows:
	\begin{align*}
		&\frac{r^2\lambda+r}{\lambda+r}<\tahq\\
		&\iff r^2\lambda+r<\tahq(\lambda+r)\\
		&\iff r^2\lambda+r\left(1-\tahq\right)-\lambda\tahq<0\\
		&\iff r^2+\frac{1-\tahq}{\lambda}r-\tahq<0.
	\end{align*}
	Note that the equation $r^2+\frac{1-\tahq}{\lambda}r-\tahq=0$ has solutions
	$$
	r_{1,2}=\frac{\tahq-1}{2\lambda}\pm\sqrt{\frac{\left(1-\tahq\right)^2}{4\lambda^2}+\tahq}.
	$$
	Then, $r_1<0$ and, since $\sqrt{a^2+b^2}\leq a+b$ for $a,b\geq 0$, it holds that
	\begin{align*}
		r_2&=\frac{\tahq-1}{2\lambda}+\sqrt{\frac{\left(1-\tahq\right)^2}{4\lambda^2}+\tahq}\\&\leq \frac{\tahq-1}{2\lambda}+\frac{1-\tahq}{2\lambda}+\tah=\tah.
	\end{align*}
	Finally, with the fact that  $H(r,u)\cap\left[0, \tanh(\frac{h}{2})e_1\right]\neq\emptyset$ is only possible for $r\in[0,\tah],$ the lemma follows by using the equations above for $\lambda\in(0,1]$ to rewrite $r\in[0,\tah]$ with $\frac{r^2\lambda+r}{\lambda+r}<\tahq$ as $r\in[0,r_2]$ and, similarly, $r\in[0,\tah]$ with $\frac{r^2\lambda+r}{\lambda+r}\geq \tahq$ as $r\in[r_2,\tah]$.
	
	For $\lambda=0$ we always have $r^2+2Rr=1\geq \tah$. Thus, in this case $H(r,u)\cap\left[0, \tanh(\frac{h}{2})e_1\right]\neq\emptyset$ if $r\in[0,\tah]$ and $u_1\geq f_3(r)$.
\end{proof}

For a closed set $A\subset\H^d$ let $[A]_\lambda$ denote the collection of $\lambda$-geodesic hyperplanes having non-empty intersection with $A$ and $o$ on its non-convex side, i.e.,
\begin{equation}\label{eq:F_A}
	[A]_\lambda := \{H\in \mathrm{Hyp}_\lambda^o: H\cap A\neq\emptyset\}.
\end{equation}
To determine the measure of $[A]_\lambda$, where $A$ is a geodesic segment of given length, we need a formula for the spherical Lebesgue measure of spherical caps. As before, we denote by $\mathcal{S}(u,h)=\{x\in\mathbb{S}^{d-1}:\langle u,x\rangle\geq h\}$ the closed spherical cap centred at $u\in\mathbb{S}^{d-1}$ with height $0<h<1$. Then,
\begin{equation}\label{eq:VolSphericalCap}
	\sigma_{d-1}(\mathcal{S}(u,h)) =\frac{\Gamma(\frac{d}{2})}{\sqrt{\pi}\,\Gamma(\frac{d-1}{2})}\int_h^1(1-s^2)^{d-3\over 2}\,\dint s,
\end{equation}
see \cite[Equation (2.7)]{KSTBook}. Further, we recall the notation
$$
r_c=\frac{\tahq-1}{2\lambda}+\sqrt{\frac{\left(1-\tahq\right)^2}{4\lambda^2}+\tahq}.
$$
The first main result of this subsection is the following integral representation for the intersection measure.

\begin{coro}\label{cor:IntRepInvMeasure}
	Let $h>0$ and $\ell(h)$ be an arbitrary geodesic segment of length $h$ starting at $o$.
	\begin{itemize}
		\item[(a)] If $\lambda=0$ then
		\begin{align*}
			\mu_{\gamma,0}([{\ell(h)}]_0)&=\gamma\int_{0}^{\tah} \sigma_{d-1}\left(\mathcal{S}\left(e_1,\frac{\tahq r+r}{\tah(r^2+1)}\right)\right)\frac{2(1+r^2)^{d-1}}{(1-r^2)^d}\;\mathrm{d}r.
		\end{align*}
		
		\item[(b)] If $0<\lambda\leq 1$ then
		\begin{align*}
			&\mu_{\gamma,\lambda}([{\ell(h)}]_\lambda)
			=\gamma\int_0^{r_c} \sigma_{d-1}\left(\mathcal{S}\left(e_1,\frac{2\sqrt{(r^2\lambda+r)(\lambda+r)}}{r^2+2\lambda r+1}\right)\right)\frac{2(1+2\lambda r+r^2)^{d-1}}{(1-r^2)^d}\;\mathrm{d}r\\
			& \quad +\gamma\int_{r_c}^{\tah} \sigma_{d-1}\left(\mathcal{S}\left(e_1,\frac{\tahq(\lambda+r)+r^2\lambda+r}{\tah(r^2+2\lambda r+1)}\right)\right)\frac{2(1+2\lambda r+r^2)^{d-1}}{(1-r^2)^d}\;\mathrm{d}r.
		\end{align*}
	\end{itemize}
\end{coro}
\begin{proof}
	For $0\leq\lambda\leq 1$ and by \eqref{eq:invariant_measure_poincare_disc_model} we have that $\mu_{\gamma,\lambda}([{\ell(h)}]_\lambda)$ is the same as
	$$
	\gamma\int_0^1\int_{S^{d-1}}{\bf 1}\Big\{H(r,u)\cap\Big[0, \tanh\Big(\frac{h}{2}\Big)e_1\Big]\neq\emptyset\Big\}\frac{2(1+2\lambda r+r^2)^{d-1}}{(1-r^2)^d}\;\sigma_{d-1}(\mathrm{d}u)\mathrm{d}r.
	$$
	The result follows from this and Lemma \ref{lemma:schnitt_indicator}.
\end{proof}

\subsection{Evaluation of the integrals for $\lambda=0$ and $\lambda=1$}

We were not able to analytically evaluate the integrals describing $\mu_{\gamma,\lambda}([\ell(h)]_\lambda)$ in Corollary \ref{cor:IntRepInvMeasure} for general $0\leq\lambda\leq 1$. However, for the two special cases $\lambda=0$ and $\lambda=1$ such an evaluation is possible as we shall demonstrate in this subsection. We start with the case $\lambda=0$.

\begin{lemma}\label{lem:Lambda=0Explicitly}
	Suppose $\lambda=0$. Then
	$$
	\mu_{\gamma,0}([{\ell(h)}]_0)  = \gamma\,\frac{\Gamma(\frac{d}{2})}{2\sqrt{\pi}\,\Gamma(\frac{d+1}{2})}\,h,\qquad h\geq 0.
	$$
\end{lemma}
\begin{proof}
	Recall from Corollary \ref{cor:IntRepInvMeasure} (a) that for $\lambda=0$ we have
	\begin{align*}
		\mu_{\gamma,0}([{\ell(h)}]_0)&=\gamma\int_{0}^{\tah} \sigma_{d-1}\left(\mathcal{S}\left(e_1,\frac{\tahq r+r}{\tah(r^2+1)}\right)\right)\frac{2(1+r^2)^{d-1}}{(1-r^2)^d}\;\mathrm{d}r.
	\end{align*}
	We apply the substitution $r=\tanh{x\over 2}$ and simplify. This yields
	\begin{align*}
		\mu_{\gamma,0}([{\ell(h)}]_0)&=\gamma\int_0^h\sigma_{d-1}\Big(\mathcal{S}\Big(e_1,{\tanh x\over\tanh h}\Big)\Big)\,\cosh^{d-1}x\,\dint x.
	\end{align*}
	Now, take the derivative with respect to $h$ and use the integral representation \eqref{eq:VolSphericalCap} for $\sigma_{d-1}(\mathcal{S}(u,h))$. This gives
	\begin{align*}
		{\dint\over\dint h}\mu_{\gamma,0}([{\ell(h)}]_0) = \gamma\frac{\Gamma(\frac{d}{2})}{\sqrt{\pi}\,\Gamma(\frac{d-1}{2})}\int_0^h\Big(1-{\tanh^2x\over\tanh^2h}\Big)^{d-3\over 2}{\tanh x\over\sinh^2h}\,\cosh^{d-1}x\,\dint x.
	\end{align*}
	Next, we applying the substitution $u={\tanh^2x\over\tanh^2h}$ and then use the Euler integral \cite[Equation (15.6.1)]{NIST}:
	\begin{align*}
		{\dint\over\dint h}\mu_{\gamma,0}(\mathcal{F}_{\ell(h)}) &= \gamma\frac{\Gamma(\frac{d}{2})}{\sqrt{\pi}\,\Gamma(\frac{d-1}{2})}{\tanh^2h\over 2\sinh^2 h}\int_0^1(1-u)^{d-3\over 2}(1-u\tanh^2h)^{-{d+1\over 2}}\,\dint u\\
		&=\gamma\frac{\Gamma(\frac{d}{2})}{\sqrt{\pi}\,\Gamma(\frac{d-1}{2})}{\tanh^2h\over 2\sinh^2 h}B\Big(1,{d-1\over 2}\Big){_2F_1}\Big({d+1\over 2},1;{d+1\over 2};\tanh^2h\Big)
	\end{align*}
	with the beta and the Gau\ss\ hypergeometric function. We have $B(1,{d-1\over 2})={2\over d-1}$ and since the first and the third argument of the hypergeometric function are the same, ${_2F_1(a,b;a;z)=(1-z)^{-b}}$ by \cite[Equation (15.4.6)]{NIST}, and this term simplifies to $1/(1-\tanh^2h)$. As a result,
	$$
	{\dint\over\dint h}\mu_{\gamma,0}([{\ell(h)}]_0) = \gamma\frac{\Gamma(\frac{d}{2})}{\sqrt{\pi}\,\Gamma(\frac{d-1}{2})}{\tanh^2h\over 2\sinh^2 h}\cdot {2\over d-1}\cdot{1\over 1-\tanh^2h} = \gamma\,\frac{\Gamma(\frac{d}{2})}{2\sqrt{\pi}\,\Gamma(\frac{d+1}{2})},
	$$
	independently of $h$, implying that $\mu_{\gamma,0}([{\ell(h)}]_0)$ must be an affine-linear function in $h$. Since $\mu_{\gamma,0}([{\ell(0)}]_0) =0$, it follows that
	$$
	\mu_{\gamma,0}([{\ell(h)}]_0)  = \gamma\,\frac{\Gamma(\frac{d}{2})}{2\sqrt{\pi}\,\Gamma(\frac{d+1}{2})}\,h
	$$
	for all $h\geq 0$.
\end{proof}

Next, we deal with the case $\lambda=1$ corresponding to horospheres.

\begin{lemma}\label{lem:Lambda=1Explicitly}
	Suppose $\lambda=1$. Then
	$$
	\mu_{\gamma,1}([{\ell(h)}]_1)  = \gamma\,\frac{\Gamma(\frac{d}{2})}{2\sqrt{\pi}\,\Gamma(\frac{d+1}{2})}\,h,\qquad h\geq 0.
	$$
\end{lemma}
\begin{proof}
	We have $\mu_{\gamma,1}([{\ell(h)}]_1) = I_1(h) + I_2(h)$ with
	\begin{align*}
		I_1(h) &:=\gamma\int_0^{t}\sigma_{d-1}\Big(\mathcal{S}\Big(e_1,{2\sqrt{r}\over r+1}\Big)\Big)W(r)\,\dint r,\\
		I_2(h) &:= \gamma\int_{t^2}^{t}\sigma_{d-1}\Big(\mathcal{S}\Big(e_1,{t^2+r\over t(1+r)}\Big)\Big)W(r)\,\dint r,
	\end{align*}
	where we put $t:=\tah$ and $W(r):=2{(1+r)^{2(d-1)}\over(1-r^2)^d}$. Next, we set $z(r,t):={t^2+r\over t(1+r)}$, $B(r,t):=\sigma_{d-1}(\mathcal{S}(e_1,z(r,t)))W(r)$ and $A(r):=\sigma_{d-1}\Big(\mathcal{S}\Big(e_1,{2\sqrt{r}\over r+1}\Big)\Big)W(r)$. Then $I_1(h)$ and $I_2(h)$ can be written as a function of $t$ as
	$$
	I_1(t) = \gamma\int_0^{t^2}A(r)\,\dint r\qquad\text{and}\qquad I_2(t)=\gamma\int_{t^2}^tB(r,t)\,\dint r.
	$$
	Note that $A(t^2)=B(t^2,t)$, $z(t,t)=1$ and so $B(t,t)=0$. We can now differentiate with respect to $t$ according to the Leibniz rule:
	$$
	{\dint\over\dint t}(I_1(t)+I_2(t)) = \gamma\int_{t^2}^t\partial_t B(r,t)\,\dint r.
	$$
	By definition of $B(r,t)$, the chain rule and \eqref{eq:VolSphericalCap} we get
	$$
	\partial_t B(r,t) = -\frac{\Gamma(\frac{d}{2})}{\sqrt{\pi}\,\Gamma(\frac{d-1}{2})}(1-z(r,t)^2)^{d-3\over 2}\partial_t z(r,t)W(r).
	$$
	From $\partial_t z(r,t)={t^2-r\over t^2(1+r)}$ and the definitions of $z(r,t)$ and $W(r)$ it follows after simplification that
	$$
	\partial_t B(r,t) = -\frac{\Gamma(\frac{d}{2})}{\sqrt{\pi}\,\Gamma(\frac{d-1}{2})} {2(1-t^2)^{d-3\over 2}\over t^{d-1}}{(t^2-r^2)^{d-3\over 2}(t^2-r)\over(1-r)^d}.
	$$
	Thus,
	$$
	{\dint\over\dint t}(I_1(t)+I_2(t)) = -\gamma\frac{\Gamma(\frac{d}{2})}{\sqrt{\pi}\,\Gamma(\frac{d-1}{2})}{2(1-t^2)^{d-3\over 2}\over(d-1)t^{d-1}}\int_{t^2}^t H'(r)\,\dint r
	$$
	if we put $H(r):={(t^2-r^2)^{d-1\over 2}\over(1-r)^{d-1}}$. But
	$$
	\int_{t^2}^t H'(r)\,\dint r = H(t) - H(t^2) = 0 - {t^{d-1}\over(1-t^2)^{d-1\over 2}}
	$$
	and hence
	$$
	{\dint\over\dint t}(I_1(t)+I_2(t)) = \gamma\frac{\Gamma(\frac{d}{2})}{\sqrt{\pi}\,\Gamma(\frac{d-1}{2})}{2(1-t^2)^{d-3\over 2}\over(d-1)t^{d-1}}{t^{d-1}\over(1-t^2)^{d-1\over 2}} = \gamma\frac{\Gamma(\frac{d}{2})}{\sqrt{\pi}\,\Gamma(\frac{d-1}{2})}{2\over d-1}{1\over 1-t^2}.
	$$
	Now, we recall that $t=\tah$. Then the chain rule implies that
	\begin{align*}
	{\dint\over\dint h}(I_1(h)+I_2(h)) &= {\dint\over\dint t}(I_1(t)+I_2(t)){\dint t\over\dint h}\left(\tah\right) \\
	&= \gamma\frac{\Gamma(\frac{d}{2})}{\sqrt{\pi}\,\Gamma(\frac{d-1}{2})}{2\over d-1}{1\over 1-t^2}{1-t^2\over 2} \\
	&= \gamma\,\frac{\Gamma(\frac{d}{2})}{2\sqrt{\pi}\,\Gamma(\frac{d+1}{2})},
	\end{align*}
	implying that $\mu_{\gamma,1}([{\ell(h)}]_1)$ must be an affine-linear function of $h$. But again, since $\mu_{\gamma,1}([{\ell(0)}]_1)=0$, we conclude that
	$$
	\mu_{\gamma,1}([{\ell(h)}]_1) = \gamma\,\frac{\Gamma(\frac{d}{2})}{2\sqrt{\pi}\,\Gamma(\frac{d+1}{2})}\,h
	$$
	for all $h\geq 0$.
\end{proof}

\subsection{Linearity for all $0\leq\lambda\leq 1$}

As mentioned in the previous subsection, for $0<\lambda<1$ we were not able to analytically evaluate the integrals in the representation for $\mu_{\gamma,\lambda}([{\ell(h)}]_\lambda)$ in Corollary \ref{cor:IntRepInvMeasure}. To show that also in this case $\mu_{\gamma,\lambda}([{\ell(h)}]_\lambda)$ is a linear function and the same as for $\lambda=0$ in Lemma \ref{lem:Lambda=0Explicitly} and $\lambda=1$ in Lemma \ref{lem:Lambda=1Explicitly} we take a different route. We start with the following result.

\begin{lemma}\label{lem:Slope}
	Suppose that $0\leq\lambda\leq 1$. Then
	$$
	\lim_{h\downarrow 0}{\mu_{\gamma,\lambda}([{\ell(h)}]_\lambda)\over h} =  \gamma\,\frac{\Gamma(\frac{d}{2})}{2\sqrt{\pi}\,\Gamma(\frac{d+1}{2})}.
	$$
\end{lemma}
\begin{proof}
	Using $\tanh(x)=x-x^3/3+O(x^5)$ as $x\to 0$ one first obtains
	$$
	r_c={\lambda\over 4}h^2+{\lambda\over 48}(1-3\lambda^2)h^4+O(h^6),\qquad h\downarrow 0.
	$$
	Call $F(r)$ the integrand in the first integral $I_1(h)$ from $0$ to $r_c$ in Corollary \ref{cor:IntRepInvMeasure} (b). Then $F(0)={1\over 2}\cdot 2=1$ by \eqref{eq:VolSphericalCap}. Moreover,
	$$
	\frac{2\sqrt{(r^2\lambda+r)(\lambda+r)}}{r^2+2\lambda r+1} = 2\sqrt{\lambda r} + {(1-\lambda^2)\over\sqrt{\lambda}}r^{3/2} + O(r^{5/2}),
	$$
	implying that $F(r)=F(0)+O(\sqrt{r})$ as $r\downarrow 0$. It follows that
	$$
	{1\over\gamma}I_1(h) = \int_0^{r_c(h)}F(r)\,\dint r = F(0)r_c(h) + \int_0^{r_c(h)}F(r)-F(0)\,\dint r = r_c(h)+O(r_c(h)^{3/2}).
	$$
	Thus,
	\begin{equation}\label{eq:AsymptoticsI1}
		I_1(h) = \gamma{\lambda\over 4}h^2 + o(h^2),\qquad h\downarrow 0.
	\end{equation}
	Next, we consider the integral $I_2(h)$ from $r_c$ to $\tanh({h\over 2})$ in Corollary \ref{cor:IntRepInvMeasure} (b). Apply the substitution $r=sh$ with $r_c/h\leq s\leq\tanh({h\over 2})/h$. Then,
	$$
	{1\over\gamma}I_2(h) = h\int_{r_c/h}^{{1\over h}\tanh({h\over 2})}\sigma_{d-1}(\mathcal{S}(e_1,A(h,s)))\,B(h,s)\,\dint s
	$$
	with the abbreviations
	\begin{align*}
		A(h,s) &:= {\tahq(\lambda+sh)+s^2h^2\lambda +sh\over\tah(1+2\lambda sh+s^2h^2)},
		\qquad B(h,s) := {2(1+2\lambda sh+s^2h^2)^{d-1}\over(1-s^2h^2)^d}.
	\end{align*}
	Using the expansion of the hyperbolic tangent function from above, $A(h,s)=2s(1+O(h))$ and $B(h,s)=2+O(h)$ as $h\downarrow 0$ uniformly for $0\leq s<1/2$ (note that ${1\over h}\tah\to 1/2$, whereas $r_c/h\to 0$). It follows that
	$$
	I_2(h)  = 2\gamma h \int_0^{1/2}\sigma_{d-1}(\mathcal{S}(e_1,2s))\,\dint s+o(h).
	$$
	Now, use the formula \eqref{eq:VolSphericalCap} for $\sigma_{d-1}(\mathcal{S}(e_1,2s))$ and Fubini's theorem:
	\begin{align*}
		2\int_0^{1/2}\sigma_{d-1}(\mathcal{S}(e_1,2s))\,\dint s &= \int_0^1\sigma_{d-1}(\mathcal{S}(e_1,u))\,\dint u\\
		&=\frac{\Gamma(\frac{d}{2})}{\sqrt{\pi}\,\Gamma(\frac{d-1}{2})}\int_0^1(1-t^2)^{d-3\over 2}\Big(\int_0^t\,\dint u\Big)\,\dint t\\
		&=\frac{\Gamma(\frac{d}{2})}{\sqrt{\pi}\,\Gamma(\frac{d-1}{2})}\int_0^1(1-t^2)^{d-3\over 2}t\,\dint t\\
		&=\frac{\Gamma(\frac{d}{2})}{2\sqrt{\pi}\,\Gamma(\frac{d-1}{2})}\int_0^1(1-x)^{d-3\over 2} \,\dint x\\
		&=\frac{\Gamma(\frac{d}{2})}{2\sqrt{\pi}\,\Gamma(\frac{d+1}{2})},
	\end{align*}
	where we applied the substitutions $u=2s$ and $x=t^2$. As a result,
	$$
	I_2(h) = \gamma\,\frac{\Gamma(\frac{d}{2})}{2\sqrt{\pi}\,\Gamma(\frac{d+1}{2})}\,h+o(h).
	$$
	If we combine this with \eqref{eq:AsymptoticsI1}, divide by $h$ and take the limit as $h\downarrow 0$ we get
	$$
	\lim_{h\downarrow 0}{\mu_{\gamma,\lambda}([{\ell(h)}]_\lambda)\over h} = \lim_{h\downarrow 0}{I_1(h)+I_2(h)\over h} =  \gamma\,\frac{\Gamma(\frac{d}{2})}{2\sqrt{\pi}\,\Gamma(\frac{d+1}{2})}
	$$
	and the proof is complete.
\end{proof}

The next step consists in establishing the following fact.

\begin{lemma}\label{lem:Linear}
	Suppose that $0\leq\lambda\leq 1$. Then $\mu_{\gamma,\lambda}([{\ell(h)}]_\lambda)$ is a linear function in $h$ (with zero intercept).
\end{lemma}
\begin{proof}
Since the case $\lambda=0$ was handled in Lemma \ref{lem:Lambda=0Explicitly} and also in \cite{BuehlerHugThaele}, we restrict our attention to the case $0<\lambda\leq 1$ in what follows.
Let $e_1$ be a unit basis vector in $\R^d$. Consider the geodesic line $L=\{ c e_1: c\in (-1,1)\}$ in the Poincar\'e ball model $\mathbb{B}^d$ and let $\overline L := \{ c e_1: c\in [-1,1]\}$ be its closure. We identify $\ell(h)$ with the geodesic segment $[0, (\tanh \frac h2) e_1] \subset L$.  We say that the point $c e_1\in \overline L$ is below $de_1\in \overline L$ if $-1 \leq c \leq  d \leq  1$.

Every $\lambda$-geodesic hyperplane $H\in \mathrm{Hyp}_\lambda$  (which may have $o$ on its convex side, or not)  intersects $L$ in $0,1$ or $2$ points. Let  $Q_\lambda \subset  \mathrm{Hyp}_\lambda$  be the set of all $\lambda$-geodesic hyperplanes $H\in \mathrm{Hyp}_\lambda$ that intersect $L$ (in one or two points). If $H\in Q_\lambda$ intersects $L$ in two points, we denote these by  $C(H)\in L$ and $D(H)\in L$ with the convention that $C(H)$ is below $D(H)$.  If $H\in Q_\lambda$ is tangent to $L$, we denote the point of tangency by $C(H) = D(H)\in L$. Finally, if $H\in Q_\lambda$ intersects $L$ in precisely one point and the intersection is transversal, then the convex side of $H$  contains either $e_1$ or $-e_1$. If $e_1$ is on the convex side, we denote the unique intersection points of $H$ and $L$ by $C(H)$ and put $D(H) := e_1$. On the other hand, if $-e_1$ is on the convex side, we denote the unique intersection point of $L$ and $H$  by $D(H)$ and put $C(H) := -e_1$. 

This defines two maps $C, D: Q_\lambda \to \overline L$ such that $C(H)$ is below  $D(H)$ for every $H\in Q_\lambda$. Note that the intersection of the convex side of $H$ with $\overline L$ is the interval $[C(H), D(H)]$ on $\overline{L}$.
Observe also that $Q_\lambda$ is invariant under hyperbolic isometries $\varphi$ of $\mathbb B^{d}$ that map $L$ to $L$ and satisfy $\varphi(e_1) = e_1$, $\varphi(-e_1)= -e_1$. (The latter means that $\varphi$ is orientation-preserving on $L$. Note that every isometry admits a unique continuous extension to the closure of $\mathbb B^d$.)
The map $C: Q_\lambda \to \overline L$ commutes with all hyperbolic isometries $\varphi$ that map $L$ to $L$ and are orientation-preserving on $L$. More precisely, $C(\varphi(H)) = \varphi(C(H))$ for all $H\in Q_\lambda$. Let $m_\lambda$ be the  push-forward of the invariant measure $\nu_\lambda$ (restricted to $Q_\lambda$) under the map $C$. Then $m_\lambda$  is a measure on $\overline L$ invariant under hyperbolic orientation-preserving isometries of $L$. Also, $m_\lambda$ is locally finite on $L$ (but may assign infinite mass to $-e_1$), since for any interval $I_a$ of length $a$ centred at $o$ we have
\begin{align*}
m_\lambda(I_a) &\leq \nu_\lambda(\{H\in{\rm Hyp}_\lambda:H\cap I_a\neq\varnothing\})\leq\int_{-a/2}^{a/2}(\cosh(s)+\lambda\sinh(s))^{d-1}\,\dint s<\infty.
\end{align*}
It follows that the restriction of $m_\lambda$ to $L$  is a multiple of the hyperbolic length measure on $L$. In particular, $m_\lambda (\ell(h))$ is a linear function of $h>0$.

Next we  observe that for a hyperplane $H\in Q_\lambda$,    $C(H) \in \ell(h)$ is equivalent to  $H\in [{\ell(h)}]_\lambda$. (Recall that $H\in [{\ell(h)}]_\lambda$ means that $H\in \mathrm{Hyp}_\lambda^o$ and $H\cap \ell(h) \neq \varnothing$.) Indeed, if $C(H) \in \ell(h)$ then both $C(H)$ and $D(H)$ are ``above''  the origin and hence the origin is not on the convex side of $H$. So $H\in \mathrm{Hyp}_\lambda^o$.  Since $H$ intersects the geodesic segment $\ell(h)$ at $C(H)$, we conclude that $H\in [{\ell(h)}]_\lambda$. Conversely, if $H\in [{\ell(h)}]_\lambda$, then by definition of $[{\ell(h)}]_\lambda$,  $0$ is not on the convex side of $H$ and  $H$ intersects $\ell(h)$. This implies that $C(H)$ and $D(H)$ are both above $0$ and since $H$ intersects $\ell(h)$, we conclude $C(H) \in \ell(h)$, proving the equivalence. 

The equivalence we just proved shows that $\mu_{\gamma,\lambda}([{\ell(h)}]_\lambda) = \gamma m_\lambda (\ell(h))$. Since $m_\lambda (\ell(h))$ is linear in  $h>0$, the proof is complete.
\end{proof}

If we combine Lemma \ref{lem:Slope} with Lemma \ref{lem:Linear} we obtain the following generalization of Lemma \ref{lem:Lambda=0Explicitly} and Lemma \ref{lem:Lambda=1Explicitly}.

\begin{coro}\label{coro:mu(F_[o,l])}
	Suppose $0\leq\lambda\leq 1$. Then
	$$
	\mu_{\gamma,\lambda}([{\ell(h)}]_\lambda)  = \gamma\,\frac{\Gamma(\frac{d}{2})}{2\sqrt{\pi}\,\Gamma(\frac{d+1}{2})}\,h,\qquad h\geq 0.
	$$
\end{coro}

\section{Expected volume of the visibility region}\label{sec:ExpectedVolume}

In this section we derive the formula for the expected volume of the visibility region $Z_{\gamma,\lambda,d}$. The calculations are analogous to the ones in \cite{BuehlerHugThaele} and will eventually show that $\mathbb{E}{\rm vol}_d(Z_{\gamma,\lambda,d})$ is independent of $\lambda$. For $u\in\mathbb{S}^{d-1}$ let
$$
s_u(\eta_{\gamma,\lambda})=\sup\left\{h\geq 0: \left\{r u : 0\leq r \leq \tanh\left(\frac{h}{2}\right)\right\} \cap H =\emptyset \text{ for all }H\in\eta_{\gamma,\lambda}\right\}.
$$
Here, the set $\{r u : 0\leq r < 1\}$ represents the geodesic ray in the Poincaré ball model starting at $o$ with direction $u$, and the upper bound $\tanh(h/2)$ corresponds to the Euclidean distance of a point with hyperbolic distance $h$ from the origin.

\begin{proof}[Proof of Theorem \ref{thm:MainIntro} (b)]
	Fix $u\in\mathbb{S}^{d-1}$ and let $\ell(h)$ be a geodesic segment of length $h>0$ starting at $o$ in an arbitrary direction. Then by Corollary \ref{coro:mu(F_[o,l])} we have
	\begin{align*}
		\p(s_u(\eta_{\gamma,\lambda})>h)&=\p(\eta_{\gamma,\lambda}([\ell(h)]_\lambda)=0)
		=\exp\big({-\mu_{\gamma,\lambda}([\ell(h)]_\lambda)}\big)=\exp\Big({-\gamma\,\frac{\Gamma(\frac{d}{2})}{2\sqrt{\pi}\,\Gamma(\frac{d+1}{2})}\,h}\Big),
	\end{align*}
	independently of $\lambda$.
	Hence, the random variable $s_u(\eta_{\gamma,\lambda})$  is exponentially distributed with parameter $\gamma\,\frac{\Gamma(\frac{d}{2})}{2\sqrt{\pi}\,\Gamma(\frac{d+1}{2})}$ and by the polar decomposition of hyperbolic space we have
	\begin{align*}
		\mathbb{E}{\rm vol}_d(Z_{\gamma,\lambda,d})
		&={2\pi^{d/2}\over\Gamma({d\over 2})}\E\Bigg[\int_{\mathbb{S}^{d-1}}\int_0^{s_u(\eta_{\gamma,\lambda})}\sinh^{d-1}(s)\;\mathrm{d}s\sigma_d(\mathrm{d}u)\Bigg]\\
		&={2\pi^{d/2}\over\Gamma({d\over 2})}\int_{\mathbb{S}^{d-1}}\int_0^{\infty}\p(s_u(\eta_{\gamma,\lambda})\geq s)\sinh^{d-1}(s)\;\mathrm{d}s\sigma_d(\mathrm{d}u)\Bigg]\\
		&={2\pi^{d/2}\over\Gamma({d\over 2})}\int_0^\infty \exp\Big({-\gamma\,\frac{\Gamma(\frac{d}{2})}{2\sqrt{\pi}\,\Gamma(\frac{d+1}{2})}\,s}\Big)\sinh^{d-1}(s)\;\mathrm{d}s,
	\end{align*}
	where $\sigma_d$ denotes the normalized spherical Lebesgue measure on $\mathbb{S}^{d-1}$.
	According to Identity 3.541.1 in \cite{GradRysz} (see also \cite[Equation (6.9)]{BuehlerHugThaele}) it holds that
	\begin{equation}\label{eq:SinhExpIntegral}
	\int_0^\infty\sinh^{d-1}(s)e^{-as}\,\mathrm{d}s=\frac{(d-1)!}{2^d}\frac{\Gamma(\frac{a-d+1}{2})}{\Gamma(\frac{a+d+1}{2})},
	\end{equation}
	whenever $a>d-1$. Thus, the expected volume is finite if and only if $\gamma\,\frac{\Gamma(\frac{d}{2})}{2\sqrt{\pi}\,\Gamma(\frac{d+1}{2})}>d-1$ or $\gamma>\frac{2(d-1)\sqrt{\pi}\Gamma(\frac{d+1}{2})}{\Gamma(\frac{d}{2})}=(d-1)^2\sqrt{\pi}\frac{\Gamma(\frac{d-1}{2})}{\Gamma(\frac{d}{2})}=\gamma_{\mathrm{crit}}$ and, in this case,
	$$
	\mathbb{E}{\rm vol}_d(Z_{\gamma,\lambda,d})=	\mathbb{E}{\rm vol}_d(Z_{\gamma,0,d}) = {2^{1-d}\pi^{d/2}(d-1)!\over \Gamma({d\over 2})}{\Gamma\big({\gamma_d^\ast-d+1\over 2}\big)\over\Gamma\big({\gamma_d^\ast+d+1\over 2}\big)}\quad\text{with}\quad \gamma_d^\ast := \gamma\,{\Gamma({d\over 2})\over2\sqrt{\pi}\Gamma({d+1\over 2})}.
	$$
	To arrive at the desired expression, we simplify the factor ${2^{1-d}\pi^{d/2}(d-1)!\over \Gamma({d\over 2})}$ using the Legendre duplication formula.
	It follows that
	\begin{align*}
			{2^{1-d}\pi^{d/2}(d-1)!\over \Gamma({d\over 2})}&=2^{1-d}\pi^{d/2}(d-1)! \frac{\Gamma(\frac{d+1}{2})}{2^{1-d}\sqrt{\pi}\Gamma(d)}= \pi^{\frac{d-1}{2}} \Gamma\left(\frac{d+1}{2}\right).
	\end{align*}
%
	Plugging this into the expression for $\mathbb{E}{\rm vol}_d(Z_{\gamma,\lambda,d})$ completes the proof.
\end{proof}

\begin{rema}\rm
	Note that, compared to the constant $\gamma_d^*$ obtained in \cite[Theorem 7.1]{BuehlerHugThaele}, our expression for $\gamma_d^*$ differs by an additional factor of $\frac{1}{2}$. This is due to our choice of considering the space $\mathrm{Hyp}_\lambda^o$ rather than $\mathrm{Hyp}_\lambda$, i.e., the choice of omitting all those $\lambda$- geodesic hyperplanes whose compact side contains $o$. For the case $\lambda=0$, this restriction leads to a factor of $\frac{1}{2}$ in the measure considered on $\mathrm{Hyp}_0^o$ compared to the measure on $\mathrm{Hyp}_0$ used in \cite{BuehlerHugThaele}.
\end{rema}

	\begin{rema}\rm
	The notion of $\lambda$-geodesic hyperplanes as totally umbilical hypersurfaces of $\mathbb{H}^d$ can be extended to the case $\lambda>1$. In that situation, a $\lambda$-geodesic 'hyperplane' is a hyperbolic sphere of radius
	\[
	R_\lambda : = \operatorname{artanh}\Big(\frac{1}{\lambda}\Big).
	\]
	Hence, in this regime the visibility question in presence of a Poisson process of such spheres reduces to that in a  Boolean model of balls with fixed radius $R_\lambda$, which in turn has been studied in \cite{L96}. As for $\lambda\in[0,1]$, there exists a critical intensity $\gamma_{B,\mathrm{crit}}>0$ such that the visibility region is unbounded with positive probability if $\gamma<\gamma_{B,\mathrm{crit}}$ and almost surely bounded for $\gamma>\gamma_{B,\mathrm{crit}}$. However, a key difference to the case $\lambda\in[0,1]$ is that, for $\lambda>1$, this critical intensity depends on $\lambda$. By \cite[p.\ 447]{L96} and \cite[Remark 6.5]{BuehlerHugThaele} it is given by $$\gamma_{B,\mathrm{crit}}=\frac{(d-1)\Gamma(\frac{d+1}{2})}{\pi^{(d-1)/2}\sinh(R_\lambda)}=\frac{(d-1)\Gamma(\frac{d+1}{2})}{\pi^{(d-1)/2}(\lambda^2-1)^{(d-1)/2}}.$$ In the bounded phase, the expected visible volume is derived in \cite[Theorem 6.2]{BuehlerHugThaele}. It can be expressed in the same way as the mean visible volume in \eqref{eq:MeanVolume} by replacing $\gamma_d^*$ by $v_d^*=\gamma\frac{\pi^{(d-1)/2}}{\Gamma(\frac{d+1}{2})}\sinh(R_\lambda)=\gamma\frac{\pi^{(d-1)/2}}{\Gamma(\frac{d+1}{2})(\lambda^2-1)^{(d-1)/2}}$.
\end{rema}

\end{document}